\documentclass{amsart}
\usepackage{a4}
\usepackage{multicol}

\usepackage{microtype}

\usepackage[numbers,sort]{natbib}

\usepackage{verbatim}

\usepackage{hyperref}

\usepackage{xcolor}

\usepackage{enumitem}

\usepackage{amsmath}
\usepackage{amssymb}
\usepackage{dsfont}
\usepackage{stmaryrd}

\usepackage{longtable}
\usepackage{booktabs}

\usepackage{tikz}		
\usepackage{tikz-3dplot}
\usepackage{graphics}

\usetikzlibrary{patterns,matrix,arrows,shapes}

\usepackage{tikz-cd}
\usepackage[all]{xy}
\CompileMatrices

\newtheorem{theorem}{Theorem}[section]
\newtheorem{lemma}[theorem]{Lemma}
\newtheorem{proposition}[theorem]{Proposition}
\newtheorem{corollary}[theorem]{Corollary}

\theoremstyle{definition}

\newtheorem{definition}[theorem]{Definition}
\newtheorem{example}[theorem]{Example}

\newtheorem{construction}[theorem]{Construction}

\newtheorem{remark}[theorem]{Remark}
\theoremstyle{remark}

\newcommand\CC{{\mathbb C}}

\newcommand\TT{{\mathbb T}}
\newcommand\ZZ{{\mathbb Z}}

\newcommand\QQ{{\mathbb Q}}
\newcommand\PP{{\mathbb P}}

\newcommand\KKK{{\mathcal K}}
\newcommand\RRR{{\mathcal R}}
\newcommand\OOO{{\mathcal O}}

\newcommand\AAA{{\mathcal{A}}}

\newcommand\WDiv{\operatorname{WDiv}}

\newcommand\Eff{{\rm Eff}}

\newcommand\Ample{{\rm Ample}}
\newcommand\SAmple{{\rm SAmple}}

\newcommand\trop{{\rm trop}}
\renewcommand\div{{\rm div}}
\newcommand\Cl{{\rm Cl}}
\newcommand\Pic{{\rm Pic}}
\newcommand\conv{{\rm conv}}
\newcommand\cone{{\rm cone}}
\newcommand\Spec{{\rm Spec}}

\newcommand\id{{\rm id}}

\newcommand\Aut{{\rm Aut }}

\newcommand\codim{{\rm codim}}

\newcommand\lin{{\rm lin}}

\newcommand\im{{\rm im}}
\newcommand\V{{\rm V}}

\newcommand\Chi{{\mathbb X}}

\newcommand\bangle[1]{\langle #1 \rangle}

\begin{document}
\title[On the anticanonical complex]{On the anticanonical complex}


\author[C.~Hische]{Christoff Hische} 
\address{Mathematisches Institut, Universit\"at T\"ubingen,
Auf der Morgenstelle 10, 72076 T\"ubingen, Germany}
\email{hische@math.uni-tuebingen.de}

\author[M.~Wrobel]{Milena Wrobel} 
\address{Institut f\"ur Mathematik,
Universit\"at Oldenburg,
Carl von Ossietzky Stra\ss e 9-11,
26111 Oldenburg,
Germany}
\email{milena.wrobel@uni-oldenburg.de}

\begin{abstract}
The anticanonical complex has been introduced as a natural generalization 
of the toric Fano polytope and so far has been succesfully used 
for the study of varieties with a torus action of complexity one.
In the present article we enlarge the area of application of the
anticanonical complex to varieties with a torus action of higher
complexity, for example, general arrangement varieties.
As an application of our techniques we classify the
three-dimensional canonical Fano intrinsic quadrics with
automorphism group having a maximal torus of dimension one.
\end{abstract}
\maketitle

\section{Introduction}
The main objective of this article is to
provide a combinatorial tool for the study
of singularities of certain subvarieties
of toric varieties.
Our model case are toric Fano varieties.
These are in one-to-one correspondence with 
the so called \emph{Fano polytopes}~$A_X$.
The Fano polytope $A_X$ is determined
by the property that its boundary 
$\partial A_X$ encodes the 
discrepancies of any toric resolution
of singularities:
the discrepancy of an exceptional 
divisor corresponding to a ray $\varrho$ 
equals minus one plus the ratio of the length of the 
shortest nonzero integer vector of $\varrho$
by the length of the unique 
vector in $\varrho \cap \partial A_X$.
This allows to formulate the various 
singularity conditions of the minimal 
model program for $X$ in terms of 
lattice points inside $A_X$ and 
turns the Fano polytope into the central
combinatorial tool for the classification 
of toric Fano varieties~\cite{Ba1981,KrNi2009,Ob2007, Ka2006, Ka2010}.

A natural step beyond the toric case 
is to investigate Fano varieties 
$X$ endowed with a torus action 
$\TT \times X \to X$ of complexity one,
that means that the general $\TT$-orbit 
is of codimension one in $X$.
If $X$ is in addition a Mori dream space it
allows a natural
$\TT$-equivariant embedding $X \subseteq Z_X$ 
into a toric variety $Z_X$, constructed 
via the Cox ring of $X$; see~\cite{HaSu2010}.
In particular, one can associate a
tropical variety $\trop(X)$ with $X$
and a \emph{weakly tropical resolution}
$X' \to X$, where $X'$
is the proper transform of $X \subseteq Z_X$ 
with respect to the toric morphism $Z_X'\rightarrow Z_X$ 
given by the common refinement of the 
fan of $Z_X$ and $\trop(X)$.
The toric Fano polytope is replaced 
with the \emph{anticanonical complex} 
$\mathcal{A}$ of $X$, a polyhedral complex 
supported inside $\trop(X)$; 
see~\cite{BeHaHuNi2016}.
The variety $X$ is log terminal if
and only if $\mathcal{A}$ is bounded. 
Moreover, in the latter case, $\mathcal{A}$
is determined by its property of encoding 
the discrepancies arising from toric ambient resolutions $Z_X''\rightarrow Z_X$ of $X$ 
factoring through $Z_X' \to Z_X$ in 
full analogy to the toric case.
The anticanonical complex has been used
in~\cite{BeHaHuNi2016} to classify $\QQ$-factorial 
terminal Fano threefolds with a $2$-torus 
action and Picard number~1.

We extend this picture to Mori dream spaces $X$ with torus
action $\TT\times X\rightarrow X$ of higher complexity.
Similar as in the complexity one case, the variety $X$
comes equivariantly embedded $X \subseteq Z_X$
into a toric variety and the idea is that an
anticanonical complex for $X$ should be a
polyhedral complex $\mathcal{A} \subseteq \trop(X)$
and encoding discrepancies of $X$ via its boundary
$\partial \mathcal{A}$ as indicated above. 
Here, the central question is whether the ambient
toric resolution provides enough discrepancies.
\goodbreak

Our main result reduces this problem to an
\emph{explicit maximal orbit quotient}, that means
a certain rational morphism of the
$\TT$-action, constructed via a commutative diagram
$$ 
\xymatrix{
X 
\ar[r]
\ar@{-->}[d]_{/\TT}
& 
Z_X 
\ar@{-->}[d]^{/\TT}
\\
Y
\ar[r]
&
Z_Y,
}
$$
where $Z_X \dasharrow Z_Y$ is defined on the union
over all toric orbits of codimension at most
one and there yields a categorical
quotient for the $\TT$-action;
see Construction~\ref{constr:MOQ} for the details.
The necessary property of the quotient space
$Y \subseteq Z_Y$ is
{\em semi-locally toric weakly tropical resolvability},
meaning that the weakly tropical resolution
$Y' \subseteq Z_{Y'}$ is locally toric in
a strong sense, reflecting properties of its
ambient toric variety;
see Definition~\ref{def:locallyToric}.
Our main result is the following.

\begin{theorem}\label{introthm1}
Let $X$ be a $\QQ$-Gorenstein Mori dream space with torus action
having an explicit maximal orbit quotient $X\dashrightarrow Y$,
where $Y$ is complete and admits a semi-locally toric weakly
tropical resolution.
Then $X$ admits an anticanonical complex~$\mathcal{A}$
and the following statements hold:
\begin{enumerate}
\item 
$X$ has at most log terminal singularities if and only if the
 anticanonical complex $\mathcal{A}$  is bounded.
\item
$X$ has at most canonical singularities
if and only if $0$ is the only 
lattice point in the relative interior
of~$\mathcal{A}$.
\item
$X$ has at most terminal singularities
if and only if $0$ and the primitive generators
of the rays of the defining fan of $Z_X$
are the only lattice {points~of~$\mathcal{A}$.}
\end{enumerate}
\end{theorem}

Observe that in this theorem we do not require $X$
to be a Fano variety. In fact, the Fano property
reflects in certain convexity properties
of the anticanonical complex, see Corollary~\ref{cor:fanoConvex}
and Example~\ref{ex:RAP5}.
The main step in the proof of this theorem is to
show that if $Y$ admits a semi-locally toric
weakly tropical resolution, then also $X$ does so.
Note that our resolution of singularities via
the tropical variety is a special case of
Tevelev's strategy~\cite{Te2007}, which is
non-constructive but works in full generality.

As a sample class, we consider in
Sections~\ref{section:generalArrangementVarieties} 
and~\ref{section:structuralResultsForGeneralArrangementVarieties} the \emph{general arrangement varieties},
introduced in \cite{HaHiWr2019}.
These varieties come with a torus action of
arbitrary complexity $c$
having an explicit maximal orbit quotient $X\dashrightarrow \PP_c$ and
the critical values of the quotient map
form a general hyperplane arrangement.
Examples are the Mori dream spaces with torus 
action of complexity one. There, the maximal 
orbit quotient is a projective line and the 
set of critical values 
is a point configuration on this line.

\begin{corollary}\label{cor:genArrA}
Every $\QQ$-Gorenstein general arrangement variety 
admits an anticanonical complex.
\end{corollary}

In Section \ref{section:structuralResultsForGeneralArrangementVarieties} we explicitly describe the structure of the 
anticanonical complex for an arbitrary general arrangement variety $X$. 
As our first application we use this description to obtain bounds 
on the defining data of $X$ due to its singularity type.
Here is how this runs:
The Cox ring of $X$ is a complete intersection ring of the form
$$
\RRR(X) = \CC[T_{ij}, S_k]/ \bangle{g_1, \ldots, g_{r-c}}
$$ 
with a polynomial ring in the variables $T_{ij}$,
where $0 \le i \le r$, $1 \le j \le n_i$, and
$S_k$, where $1 \le k \le m$.
Each relation $g_t$ has $c+2$ terms, where $c$ 
is the complexity of the torus action and 
is of the form
$$
g_t 
\ = \
a_{t,0}T_0^{l_0} + 
\ldots +
a_{t,c}T_{c}^{l_{c}}
+
a_{t,c+t}T_{c+t}^{l_{c+t}},
\qquad
T_i^{l_i} 
\ = \ 
T_{i1}^{l_{i1}}\cdots T_{in_i}^{l_{in_i}},
$$ 
with $a_{t,j} \in \CC^*$. 
The variables $T_{ij}$ and $S_k$ are in correspondence 
with the rays $\varrho_{ij}$ and $\varrho_k$
of the fan $\Sigma$ of the ambient toric variety $Z_X$ 
of $X$.
For a ray $\varrho \in \Sigma$ we denote by $l_\varrho$ 
the exponent of the corresponding variable $T_\varrho$
in the above relations. Now the singularity type of $X$ provides bounds 
on the exponents $l_\varrho$ as follows.

\goodbreak

\begin{theorem}
\label{thm:3}
Let $X\subseteq Z_X$ be a $\QQ$-Gorenstein general arrangement variety of complexity $c$ 
and consider a cone $\sigma =\varrho_{0j_0} + \ldots + \varrho_{rj_r}\in\Sigma$. 
If the singularity defined by $\sigma$ is 
\begin{enumerate}
\item 
log terminal, then 
$\sum_{\varrho \in \sigma^{(1)}}{l_{\varrho}^{-1}} > r-c$ holds,
\vspace{2pt}
\item 
canonical, then 
$\sum_{\varrho \in \sigma^{(1)}}{l_{\varrho}^{-1}} \geq 
r-c + c_{\sigma}\prod_{\varrho \in \sigma^{(1)}}{l_{\varrho}^{-1}}$ holds,
\vspace{2pt}
\item 
terminal, then
$\sum_{\varrho \in \sigma^{(1)}}{l_{\varrho}^{-1}} > 
 r-c + c_{\sigma}\prod_{\varrho \in \sigma^{(1)}}{l_{\varrho}^{-1}}$ holds,
\end{enumerate}
where $c_\sigma$ is the greatest common divisor of the entries of the vector $v_\sigma$ built up from
the primitive generators $v_{ij_i}\in\varrho_{ij_i}$ as follows:
$$ 
v_\sigma
\ := \ 
\ell_{\sigma,0} v_{0j_0} + \ldots +  \ell_{\sigma,r} v_{rj_r}
\ \in \ 
\ZZ^{r+s},\qquad
\ell_{\sigma,i} 
\ := \ 
\frac{l_{0j_0} \cdots l_{rj_r}}{l_{ij_i}}\ \in\ \ZZ.
$$
\end{theorem}

As a direct consequence we obtain
the following characterization of
log terminality 
in the complexity two case, purely relying on 
the occurring exponents $l_{ij_i}$.

\begin{corollary}\label{introcor2}
Let $X \subseteq Z_X$ be a $\QQ$-Gorenstein general arrangement variety 
of complexity two.
Then $X$ is log terminal if and only if for any cone 
${\sigma = \varrho_{0 j_0} + \ldots + \varrho_{r j_r} \in \Sigma}$
we achieve by  suitably renumbering 
the involved rays that
${l_{4j_4} = \ldots =l_{rj_r}= 1}$ 
holds and the tuple
$(l_{0j_0},l_{1j_1},l_{2j_2},l_{3j_3})$
is one of the following:
\begin{enumerate}
\item
$(1,x,y,z)$,
\item
$(2,2,x,y)$,
\item
$(2,3,\leq 5,x)$, $(2,3,7,\leq 41)$, $(2,3,8, \leq 23)$, 
$(2,3,9,\leq 17)$,  $(2,3,10, \leq 14)$, 
$(2,3,11 \leq 13)$, 
\item
$(2,4,4,x)$, 
$(2,4,5,\leq 19)$, 
$(2,4,6,\leq 11)$, 
$(2,4,7,\leq 8)$,
\item
$(2,5,5,\leq 9)$, 
$(2,5,6,\leq 6)$,
\item
$(3,3,3,x)$, 
$(3,3,4,\leq 11)$, 
$(3,3,5,\leq 6)$, 
\item
$(3,4,4,5)$, 
$(3,4,4,4)$.
\end{enumerate}
\end{corollary}

Finally, we take a look at intrinsic 
quadrics.
These are varieties with a Cox ring
defined by a single quadratic relation,
introduced in~\cite{BeHa2007}; see also
\cite{Bo2011} for further work in
these varieties.
In~\cite{FaHa2017}, smooth Fano intrinsic
quadrics of low Picard number have been
studied.
We look at singular three-dimensional ones
coming with a torus action of true complexity
two, meaning that the maximal torus of  $\Aut(X)$
is of dimension one.

\begin{theorem}\label{thm:quadrics}
Every three-dimensional $\QQ$-factorial Fano intrinsic quadric having 
at most canonical singularities and automorphism group with
one-dimensional maximal torus is 
isomorphic to precisely one of the varieties $X$,
specified by its $\Cl(X)$-graded Cox ring $\RRR(X)$,
a matrix $[w_1,\ldots,w_r]$ of generator degrees and
an ample class $u\in\Cl(X)$ as follows:

{\small 
\renewcommand*{\arraystretch}{1.1}
\begin{longtable}{c|c|c|c|c}
No.&
$\RRR(X)$&
$\Cl(X)$&
$\left[w_1,\ldots,w_r\right]$&
$u$
\\
\hline
1&
{\tiny
$
\begin{array}{c}
     \CC[T_1,\ldots,T_4,S_1]
     \\
     \hline
     \bangle{T_1^2+T_2^2+T_3^2+T_4^2}
\end{array}$
}
&
$\ZZ\times\ZZ_2\times\ZZ_2\times\ZZ_2$&
{\tiny
$\left[\begin{array}{ccccc}
     1&1&1&1&2\\
    \bar 1&\bar 1&\bar 1&\bar 0&\bar 1\\
     \bar 0&\bar 0&\bar 1&\bar 0&\bar 1\\
     \bar 1&\bar 0&\bar 0&\bar 0&\bar 1
\end{array}\right]$
}&
{\tiny
$
\left[
\begin{array}{c}
4\\
\bar 0\\
\bar 0\\
\bar 0
\end{array}
\right]
$
}
\\
\hline
2&
{\tiny
$
\begin{array}{c}
     \CC[T_1,\ldots,T_5]
     \\
     \hline
     \bangle{T_1T_2+T_3^2+T_4^2+T_5^2}
\end{array}$
}
&
$\ZZ\times\ZZ_2\times\ZZ_2$&
{\tiny
$\left[\begin{array}{ccccc}
     1&1&1&1&1  
     \\
     \bar 0&\bar 0&\bar 1&\bar 1&\bar 0
     \\
     \bar 1&\bar 1&\bar 0&\bar 1&\bar 0
\end{array}\right]$
}&
{\tiny
$
\left[
\begin{array}{c}
3
\\
\bar 0
\\
\bar 1
\end{array}
\right]
$
}
\\
\hline
3&
{\tiny
$
\begin{array}{c}
     \CC[T_1,\ldots,T_5]
     \\
     \hline
     \bangle{T_1T_2+T_3^2+T_4^2+T_5^2}
\end{array}$
}
&
$\ZZ\times\ZZ_2\times\ZZ_2$&
{\tiny
$\left[\begin{array}{ccccc}
     1&3&2&2&2  
     \\
     \bar 1&\bar 1&\bar 1&\bar 1&\bar 0
     \\
     \bar 0 &\bar 0&\bar 1&\bar 0&\bar 1
\end{array}\right]$
}&
{\tiny
$
\left[
\begin{array}{c}
6
\\
\bar0
\\
\bar0
\end{array}
\right]
$
}
\\
\hline
4&
{\tiny
$
\begin{array}{c}
     \CC[T_1,\ldots,T_5]
     \\
     \hline
     \bangle{T_1T_2+T_3^2+T_4^2+T_5^2}
\end{array}$
}
&
$\ZZ\times\ZZ_2\times\ZZ_6$&
{\tiny
$\left[\begin{array}{ccccc}
     1&1&1&1&1  
     \\
     \bar 1&\bar 1&\bar 1&\bar 0&\bar 0
     \\
     \bar 2&\bar 4&\bar 3&\bar 3&\bar 0
\end{array}\right]$
}&
{\tiny
$
\left[
\begin{array}{c}
3
\\
1
\\
0
\end{array}
\right]
$
}
\\
\hline
5&
{\tiny
$
\begin{array}{c}
     \CC[T_1,\ldots,T_4,S_1,S_2]
     \\
     \hline
     \bangle{T_1^2+T_2^2+T_3^2+T_4^2}
\end{array}$
}
&
$\ZZ^2\times\ZZ_2\times\ZZ_2\times\ZZ_2$&
{\tiny
$\left[\begin{array}{cccccc}
    1&1&1&1&0&0\\
    0&0&0&0&1&1\\
    \bar 1&\bar 1&\bar 1&\bar 0&\bar 1&\bar 0\\
    \bar 0&\bar 0&\bar 1&\bar 0&\bar 1&\bar 0\\
    \bar 1&\bar 0&\bar 0&\bar 0&\bar 1&\bar 0
\end{array}\right]$
}&
{\tiny
$
\left[
\begin{array}{c}
2
\\
2
\\
\bar 0
\\
\bar 0
\\
\bar 0
\end{array}
\right]
$
}
\\
\hline
6&
{\tiny
$
\begin{array}{c}
     \CC[T_1,\ldots,T_5,S_1]
     \\
     \hline
     \bangle{T_1T_2+T_3^2+T_4^2+T_5^2}
\end{array}$
}
&
$\ZZ^2\times\ZZ_2\times\ZZ_2$&
{\tiny
$\left[\begin{array}{cccccc}
     1&1&1&1&1&0\\
     -1&1&0&0&0&1\\
     \bar 0&\bar 0&\bar 1&\bar 1&\bar 0&\bar 0\\
     \bar 1&\bar 1&\bar 0&\bar 1&\bar 0&\bar 0
\end{array}\right]$
}&
{\tiny
$
\left[
\begin{array}{c}
3
\\
1
\\
\bar 0
\\
\bar 1
\end{array}
\right]
$
}
\\
\hline
7&
{\tiny
$
\begin{array}{c}
     \CC[T_1,\ldots,T_5,S_1]
     \\
     \hline
     \bangle{T_1T_2+T_3^2+T_4^2+T_5^2}
\end{array}$
}
&
$\ZZ^2\times\ZZ_2\times\ZZ_2$&
{\tiny
$\left[\begin{array}{cccccc}
    -1&1&0&0&0&1\\
    2&0&1&1&1&1\\
    \bar 0&\bar 0&\bar 1&\bar 1&\bar 0&\bar 0\\
    \bar 1&\bar 1&\bar 0&\bar 1&\bar 0&\bar 0
     \end{array}\right]$
}&
{\tiny
$
\left[
\begin{array}{c}
1
\\
4
\\
\bar 0
\\
\bar 1
\end{array}
\right]
$
}
\\
\hline
8&
{\tiny
$
\begin{array}{c}
     \CC[T_1,\ldots,T_5,S_1]
     \\
     \hline
     \bangle{T_1T_2+T_3^2+T_4^2+T_5^2}
\end{array}$
}
&
$\ZZ^2\times\ZZ_2\times\ZZ_2$&
{\tiny
$\left[\begin{array}{cccccc}
    1&1&1&1&1&1\\
    1&-1&0&0&0&-2\\
    \bar 0&\bar 0&\bar 1&\bar 1&\bar 0&\bar 0\\
    \bar 0&\bar 0&\bar 1&\bar 0&\bar 1&\bar 0
     \end{array}\right]$
}&
{\tiny
$
\left[
\begin{array}{c}
4
\\
-2
\\
\bar 0
\\
\bar 0
\end{array}
\right]
$
}
\\
\hline
9&
{\tiny
$
\begin{array}{c}
     \CC[T_1,\ldots,T_5,S_1,S_2]
     \\
     \hline
     \bangle{T_1T_2+T_3^2+T_4^2+T_5^2}
\end{array}$
}
&
$\ZZ^3\times\ZZ_2\times\ZZ_2$&
{\tiny
$\left[\begin{array}{ccccccc}
    1&1&1&1&1&0&0\\
    -1&1&0&0&0&1&0\\
    0&0&0&0&0&1&1\\
    \bar 0&\bar 0&\bar 1&\bar 1&\bar 0&\bar 0&\bar 0\\
    \bar 1&\bar 1&\bar 0&\bar 1&\bar 0&\bar 0&\bar 0
\end{array}\right]$
}&
{\tiny
$
\left[
\begin{array}{c}
3
\\
1
\\
2
\\
\bar 0
\\
\bar 1
\end{array}
\right]
$
}
\\
\hline
\end{longtable}
}
\end{theorem}

On our way proving Theorem $\ref{thm:quadrics}$, we obtain effective bounds for the Picard number of a $\QQ$-factorial Fano intrinsic quadric $X$, see Section \ref{sec:quadrics}. As an application we obtain in Proposition \ref{prop:quadricsTerminal} that in any dimension, there are no $\QQ$-factorial Fano intrinsic quadrics $X$ of true complexity $\dim(X)-1$ having at most terminal singularities.

\tableofcontents

\section{Toric ambient resolutions of singularities}\label{sec:toricAmbientResolutions}
In toric geometry resolution of singularities can be performed in a purely combinatorial manner. 
The idea of toric ambient resolutions of singularities is to make this methods accessible for closed subvarieties of toric varieties. The aim of this section is to give a sufficient criterion on an embedded variety for the existence of a toric ambient resolution of singularities, see Proposition \ref{prop:ambRes}.

Let us fix our terminology:
Consider a toric variety $Z$ with acting torus $T$ and a normal closed subvariety $X \subseteq Z$. Let $\varphi \colon Z' \rightarrow Z$ be a birational toric morphism. 
The \emph{proper transform} of $X \subseteq Z$ is the closure $X' \subseteq Z'$ of $\varphi^{-1}(X \cap T)$. We call $\varphi \colon Z' \rightarrow Z$ a \emph{toric ambient modification} if it maps $X'$ properly onto $X$. If furthermore the proper transform $X'$ is smooth, we call 
$\varphi$ a \emph{toric ambient resolution of singularities} of $X$.

Our approach to toric ambient resolutions of singularities is the following two-step procedure: at first we use methods from tropical geometry to prepare the embedded variety for resolving their singularities in a second step with methods from toric geometry.

Let us recall the basic notions on tropical varieties.
For a closed subvariety $X\subseteq Z$ intersecting the torus non trivially consider the vanishing ideal $I(X \cap T)$ in the Laurent polynomial ring $\OOO(T)$. For every $f \in I(X\cap T)$ let $|\Sigma(f)|$ denote the support of the codimension one skeleton of the normal quasifan of its Newton polytope, 
where a quasifan is a fan, 
where we allow the cones to be non-pointed.
Then the \emph{tropical variety $\trop(X)$ of} $X$ is defined as follows, 
see \cite[Def. 3.2.1]{MaSt2015}:
$$
\trop(X) := \bigcap_{f \in I(X \cap T)} |\Sigma(f)| \subseteq \QQ^{\dim(Z)}.
$$
A closed subvariety $X \subseteq Z$
is called \emph{weakly tropical} 
if the fan $\Sigma$ corresponding to $Z$ is supported on $\trop(X)$. 
In the following we will always assume $\trop(X)$
to be endowed with a fixed quasifan structure.
If $X \subseteq Z$ is weakly tropical then by sufficiently refining the quasifan structure fixed on $\trop(X)$ we achieve that $\Sigma$ is a subfan of $\trop(X)$.

\begin{construction}
Let $X \subseteq Z$ be a closed subvariety intersecting the torus non-trivially, consider the defining fan $\Sigma$ of $Z$ and the coarsest common refinement 
$$\Sigma':= {\Sigma \sqcap \trop(X) := \left\{\sigma \cap \tau; \ \sigma \in \Sigma, \ \tau \in \trop(X)\right\}}.$$
Let
$\varphi\colon Z'\rightarrow Z$ be the toric morphism
arising from the refinement of fans
${\Sigma' \rightarrow \Sigma}$ and
let $X'$ be the proper transform of $X$ under $\varphi$. We call $Z'\rightarrow Z$ a 
{\em weakly tropical resolution of $X$}. 
\end{construction}

Let $Z' \rightarrow Z$ be a weakly tropical resolution of $X$. Then the embedding $X' \subseteq Z'$ is weakly tropical as by construction $|\trop(X)| = |\trop(X')|$ holds.
Note that $Z'$ and thus $X'$ depend on the
choice of the quasifan structure fixed on $\trop(X)$.

For a toric variety $Z$ we denote by $Z_{\sigma}\subseteq Z$ the affine toric chart corresponding to the cone $\sigma \in \Sigma$ in the lattice $N$. 
We will make use of the local product structure of toric varieties:

\begin{construction}
\label{constr:decomp}
Let $X\subseteq Z$ be weakly tropical and let $\sigma \in \Sigma$ be any cone. 
Choose a maximal cone $\tau \in \trop(X)$ with $\sigma \preceq \tau$, set 
$N(\tau):= N\cap \lin_\QQ(\tau)$ 
and fix a decomposition $N = N(\tau) \oplus \tilde{N}$. Accordingly, we obtain a product decomposition
$$Z_\sigma \cong U(\sigma) \times \tilde\TT,$$
where $U(\sigma) := U(\sigma, \tau)$ is the affine toric variety corresponding to the cone $\sigma$ in the lattice $N(\tau)$
and $\tilde\TT $ is a torus.
We write
$\pi_{\sigma} := \pi_{\sigma, \tau, \tilde{N}}$ for the projection $Z_{\sigma} \rightarrow U(\sigma)$.
\end{construction}

Due to the structure theorem for tropical varieties, the maximal cones $\tau \in \trop(X)$ are of dimension $\dim(X)$. In particular,
in the situation of the above construction,
$U(\sigma)$ does up to isomorphism not depend on the choices made.
Moreover, if $X \subseteq Z$ is complete then due to \cite[Prop. 6.4.7]{MaSt2015} we have $|\Sigma| = |\trop(X)|$ and for any maximal cone $\sigma \in \Sigma$ the cone $\tau \in \trop(X)$ as chosen above equals $\sigma$.

\begin{definition}\label{def:locallyToric}
Let $X\subseteq Z$ be weakly tropical.
We call $X\subseteq Z$ {\em semi-locally toric} if
for every maximal cone $\sigma \in \Sigma$ there exists a projection $\pi_{\sigma}$ as in Construction \ref{constr:decomp} that maps $X_{\sigma}:=X \cap Z_{\sigma}$ isomorphically onto its image $\pi_{\sigma}(X_\sigma)$ and the latter is an open subvariety of $U(\sigma)$.
\end{definition}

\begin{remark}
Let $X$ be a rational $\TT$-variety of complexity one.
Then $X$ allows an equivariant embedding into a toric variety $Z$ such that any weakly tropical resolution is semi-locally toric, see \cite[Prop. 3.4.4.6]{ArDeHaLa2015}.
\end{remark}

Note that given a weakly tropical embedding $X \subseteq Z$ the notion of being semi-locally toric is preserved when passing over to another embedding $X' \subseteq Z'$ provided by a toric isomorphism $Z \rightarrow Z'$ as made precise in the following remark:

\begin{remark}\label{rem:sublattice}
Let $A \colon N_1 \rightarrow N_2$ be an isomorphism of lattices (we use the letter $A$ as well to denote the induced linear map $N_1\otimes\QQ \rightarrow N_2\otimes\QQ$) defining an isomorphism of affine toric varieties $\varphi_A \colon Z_1 \rightarrow Z_2$, with defining cones $\sigma_1$ in $N_1$ and $\sigma_2:= A(\sigma_1)$ in $N_2$.
Let $X \subseteq Z_1$ be a semi-locally toric closed subvariety and $N_1 = N_1(\sigma) \oplus \tilde {N_1}$ be a decomposition of $N_1$ 
such that the
corresponding projection $\pi_{\sigma} \colon Z_1 \rightarrow U(\sigma_1)$ maps $X$ isomorphically onto an open subset of $U(\sigma_1)$. Then 
$A(\trop(X)) = \trop(\varphi_A(X))$ holds and
the decomposition $N_2 = A(N_1(\sigma)) \oplus A( \tilde{N_1})$ corresponds to a projection
$\pi_{\sigma_2} \colon Z_2 \rightarrow U(\sigma_2)$ mapping $\varphi_A(X_{\sigma_1})$ isomorphically onto an open subset of $U(\sigma_2)$.
\end{remark}

\begin{proposition}\label{prop:ambRes}
Let $X\subseteq Z$ be a closed subvariety admitting a semi-locally toric weakly tropical resolution, meaning $X'\subseteq Z'$ is semi-locally toric.
Then $X\subseteq Z$ admits a toric ambient resolution of singularities.
\end{proposition}

The rest of this section is dedicated to the proof of Proposition \ref{prop:ambRes}.
Below (and in the rest of this article) we will make frequent use of the following criterion, to which we will refer to as \emph{Tevelev's criterion}, see {\cite[Lem. 2.2]{Te2007} and \cite[Thm. 6.3.4]{ MaSt2015}}:

\begin{remark}
Let $X\subseteq Z$ be a closed embedding.
Then $X$ intersects the torus orbit $T \cdot z_\sigma$ corresponding to the cone ${\sigma \in \Sigma}$ non-trivially if and only if the relative interior $\sigma^\circ$ intersects the tropical variety $\trop(X)$ non-trivially.
Moreover, if $X \subseteq Z$ is weakly tropical, then the intersection
$T\cdot z_{\sigma} \cap X$
is pure of dimension ${\dim(X)-\dim(\sigma)}.$
\end{remark}

\begin{lemma}\label{lem:weaklyTropicalIsProper}
Let $X \subseteq Z$ be a closed embedding. Then any
weakly tropical resolution $\varphi\colon Z'\rightarrow Z$ is a toric ambient modification.
\end{lemma}
\begin{proof}
We have to show that $\varphi$ maps $X'$ properly onto $X$.
Consider any completion $\trop(X)^c$ of 
the quasifan $\trop(X)$, i.e.\ a quasifan $\trop(X)^c$ with support $|\trop(X)^c| = \QQ^{\dim(Z)}$ such that $\trop(X)$ is a subfan of $\trop(X)^c$. 
Then the morphism of fans
$\Sigma \sqcap \trop(X)^c \rightarrow \Sigma$
defines a proper morphism of toric varieties
$\tilde\varphi \colon Z'' \rightarrow Z$
with $\tilde{\varphi}|_{Z'} = \varphi$,
where we regard $Z'$ as an open subset of $Z''$. We conclude that $\varphi \colon X' \rightarrow X$ is proper as Tevelev's criterion implies
$$X' = \overline{\varphi^{-1}(X\cap T)}^{Z'} = \overline{\tilde\varphi^{-1}(X\cap T)}^{Z''}.$$
\end{proof}

Let $\Sigma$ be a fan. We say that a fan $\Sigma'$ in the same lattice is a \emph{subdivision} of $\Sigma$ if
any cone $\sigma'\in\Sigma'$ is contained in a cone $\sigma\in\Sigma$ and $|\Sigma| = |\Sigma'|$ holds. Any subdivision of fans $\Sigma'\rightarrow\Sigma$ defines a proper birational morphism of the corresponding toric varieties $Z'\rightarrow Z$.

\begin{lemma}
\label{lem:resolveLocallyToric}
Let $X \subseteq Z$ be a semi-locally toric  weakly tropical embedding.
Consider
a proper birational toric morphism $\psi\colon Z'\rightarrow Z$ defined by a subdivision of fans $\Sigma'\rightarrow \Sigma$.
For $\sigma \in \Sigma$ denote by 
$\pi_{\sigma} \colon Z_{\sigma} \rightarrow U(\sigma)$ the projection mapping $X_{\sigma}$ isomorphically onto an open subvariety of $U(\sigma)$ 
and consider the morphism of toric varieties 
$\psi(\sigma)\colon V(\sigma) \rightarrow U(\sigma)$ arising via the 
subdivision of the cone $\sigma$ in $N(\sigma)$.
Then there is a commutative diagram
\begin{center}
  \begin{tikzcd}
    Z'_{\sigma}\arrow[d,"\psi",xshift=-14mm]
    \quad \cong \quad
    V(\sigma)\times\tilde{\TT}
    \arrow[d, "\psi(\sigma)\times\id", xshift=8mm]
    \arrow[r,"\pi'_{\sigma}"]&
    V(\sigma)\arrow[d, "\psi(\sigma)"]
    \\
    Z_{\sigma}
    \quad \cong\quad
    U(\sigma)\times\tilde{\TT}\arrow[r,"\pi_{\sigma}"]
    & 
    U(\sigma),
    \end{tikzcd}
\end{center}
where we set $Z'_{\sigma} := \psi^{-1}(Z_{\sigma})$, and $\pi'_{\sigma}$ maps $X' \cap Z'_{\sigma}$ isomorphically onto an open subvariety of $V(\sigma)$. 
Moreover, the following statements hold:
\begin{enumerate}
\item
The proper transform $X'$ with respect to $\psi$ equals $\psi^{-1}(X)$.
\item 
The subvariety $X'\subseteq Z'$ is weakly tropical
and semi-locally toric.
\end{enumerate}
\end{lemma}
\begin{proof}
Let $N = N(\sigma) \oplus \tilde{N}$ be the decomposition of $N$ giving rise to the isomorphism 
$Z_{\sigma} \cong U(\sigma) \times \tilde{\TT}$ and the corresponding projection $\pi_{\sigma}$. 
Then by construction the defining fan of $Z'_{\sigma}$
is supported in $N(\sigma) \otimes \QQ$. Using the same decomposition of $N$ as above we thus obtain an isomorphism
$Z'_{\sigma} \cong V(\sigma) \times \tilde \TT$,
the corresponding projection $\pi_{\sigma}'$
and a commutative diagram as claimed. 
As $\pi_{\sigma}$ maps $X_{\sigma}$ isomorphically onto its image we conclude that the projection
$\pi'_{\sigma}$ restricts to an isomorphism
$$\pi'_{\sigma}\colon \psi^{-1}(X_{\sigma}) \rightarrow \psi(\sigma)^{-1}(\pi_{\sigma}(X_{\sigma})).$$
We show that $\psi^{-1}(X_{\sigma})= X'\cap Z'_{\sigma}$ holds:
As $\psi(\sigma)^{-1}(\pi_\sigma(X_{\sigma}))$ is irreducible, so is $\psi^{-1}(X_{\sigma})$.
Thus using that $X' \cap Z'_{\sigma} \subseteq \psi^{-1}(X_{\sigma})$ is a closed irreducible subvariety of the same dimension we obtain equality as claimed.
This proves Supplement (i) and the assertion as $\psi(\sigma)^{-1}(\pi_{\sigma}(X_{\sigma}))$ is an open subvariety of $V(\sigma)$. Supplement (ii) follows by restricting the toric projection $\pi'_{\sigma}$ to the affine toric charts. 
\end{proof}

\begin{proof}[Proof of Proposition \ref{prop:ambRes}]
Let $\varphi \colon Z' \rightarrow Z$ be a semi-locally toric weakly tropical resolution. Then due to Lemma \ref{lem:weaklyTropicalIsProper} the morphism $\varphi$ maps $X'$ properly onto $X$. 
Now let $\psi \colon Z'' \rightarrow Z'$ be any toric resolution of singularities of $Z'$
arising via a regular subdivision of its defining fan $\Sigma'$
and denote by $X''$ the proper transform of $X'$ with respect to $\psi$. 
Then $\varphi\circ \psi \colon X'' \rightarrow X$ is the composition of proper morphisms and hence is proper. 
Moreover, as 
$Z''$ is smooth, the toric varieties $U(\sigma'')$, where $\sigma'' \in \Sigma''$, are smooth. In particular, as 
$X''$ is semi-locally toric due to Lemma~\ref{lem:resolveLocallyToric}~(ii), 
it is smooth as well
and we conclude that $\varphi \circ \psi$ is a toric ambient resolution of singularities. 
\end{proof}

\section{Anticanonical complexes}\label{sec:anticanonicalComplexes}
Let $Z$ be a toric variety and $X \subseteq Z$ be a closed subvariety. 
Denote by $Z_0 \subseteq Z$ the (open) 
union of all $T$-orbits of codimension 
at most one in~$Z$.
Assume that $X$ intersects $T \subseteq Z$ 
and that $X_0 := X \cap Z_0$ has a 
complement of codimension at least two in 
$X$. Then we obtain a pullback homomorphism
$$ 
\WDiv^T(Z) \ \to \ \WDiv(X), 
\qquad
D \ \mapsto \ D\vert_X,
$$
which, given a $T$-invariant Weil divisor on $Z$,
first restricts to the smooth $Z_0$, then pulls 
back to $X_0$ and finally extends to $X$ by closing 
components.
In this situation, we call $X \subseteq Z$ \emph{adapted}
if for every $T$-invariant prime divisor $D$ on $Z$, the pullback $D|_X$ is a prime divisor on~$X$. 

Assume $X \subseteq Z$ to be adapted. 
Let $\varphi \colon Z' \to Z$ be a toric ambient modification, arising from a refinement of fans $\Sigma' \rightarrow \Sigma$ in a lattice $N$, 
meaning that every $\sigma'\in \Sigma'$ is contained in some $\sigma \in \Sigma$.
We call
$\varphi$ an 
\emph{adapted toric ambient modification}
if besides $X \subseteq Z$ also $X' \subseteq Z'$ is adapted (requiring in particular~$X'$ to 
be normal). 

Let $X$ be a normal $\QQ$-Gorenstein variety. Recall that given any proper birational morphism $\varphi \colon X' \to X$ 
with a normal variety $X'$ and a canonical 
divisor $k_{X'}$ on $X'$, we have the ramification formula
$$ 
k_{X'} - \varphi^*\varphi_* k_{X'} 
\ = \ 
\sum a_E E,
$$
where $E$ runs through the exceptional prime divisors of $X' \rightarrow X$.
The number $\mathrm{discr}_X(E) := a_E \in \QQ$
is called the~\emph{discrepancy} of $X$ with respect 
to $E$; it doesn't depend on the choice of~$k_{X'}$ 
and, identifying $E$ with the local ring 
$\mathcal{O}_{X,E} \subseteq \CC(X)$,
it depends not even on 
the choice of
$\varphi \colon X' \to X$.

\begin{definition}
\label{def:antiCanRegion}
Let $X$ be $\QQ$-Gorenstein and $X \subseteq Z$ be an adapted
embedding.
Assume that the following conditions hold:
\begin{enumerate}
    \item The weakly tropical resolution $\varphi \colon Z' \rightarrow Z$ is adapted.
    \item Every proper birational toric morphism $Z'' \rightarrow Z'$ is adapted. 
    \item  There exists at least one toric ambient resolution of singularities $Z''\rightarrow Z'$.
\end{enumerate}
Then for every ray $\varrho \subseteq |\Sigma'|$ there exists a proper toric morphism $\psi \colon Z''\rightarrow Z'$ with $\varrho \in \Sigma''$.
Denote by $D_{Z''}^{\varrho}$ the corresponding toric divisor and set 
$$a_{\varrho} := \mathrm{disc}_X(D_{Z''}^{\varrho}|_{X''}).$$
Let $v_\varrho$ denote the primitive ray generator of $\varrho$ and for $a_{\varrho} > -1$ set $v_\varrho' := \frac{1}{a_{\varrho} +1 }v_{\varrho}$.
The \emph{anticanonical region} of $X \subseteq Z$ is the set
$$
\mathcal{A} := \bigcup_{\small \varrho \ \! \subseteq \ \! |\Sigma'|} \mathcal{A}_\varrho, \qquad
\mathcal{A}_{\varrho} 
:=
\begin{cases}
\conv(0, v_\varrho')
,& \text{ if } a_{\varrho} > -1
\\
\varrho,& \text{ else. } 
\end{cases}
$$
\end{definition}
In the situation of Definition \ref{def:antiCanRegion} let $Z'' \rightarrow Z$ be a toric ambient resolution of singularities
factorizing over the weakly tropical resolution $Z' \rightarrow Z$. 
Then the exceptional divisors of $X''\rightarrow X$ are precisely the pullbacks of the exceptional divisors of $Z''\rightarrow Z$ and we obtain the following:

\begin{remark}
\label{rem:anticanRegionSingTypes}
Let $X\subseteq Z$ be as in Definition \ref{def:antiCanRegion} and let $\mathcal{A}$ be the anticanonical region. 
Then the following statements hold:
\begin{enumerate}
    \item $X$ has at most log terminal singularities if and only if the anticanonical region $\mathcal{A}$ contains no ray.
    \item $X$ has at most canonical singularities if and only if for every ray $\varrho \subseteq |\Sigma'|$ we have $\varrho \cap \mathcal{A} \subseteq \conv(0, v_{\varrho})$, i.e.\ $||v_{\varrho}'|| \leq ||v_{\varrho}||$.
    \item $X$ has at most terminal singularities if and only if for every ray $\varrho \subseteq |\Sigma'|$ with $\varrho \notin \Sigma$ we have $\varrho \cap \mathcal{A} \subsetneq \conv(0, v_{\varrho})$, i.e.\ $||v_{\varrho}'|| < ||v_{\varrho}||$.
\end{enumerate}
\end{remark}

\begin{remark}
If $Z$ is a $\QQ$-Gorenstein toric variety
arising from a fan $\Sigma$, 
then for each $\sigma \in \Sigma$ there 
is a rational linear form $u_\sigma$ such 
that the anticanonical divisor of $Z$ is 
given on the affine chart $Z_\sigma \subseteq Z$ 
by $\div(\chi^{u_\sigma}) := m^{-1} \div(\chi^{mu})$, where $m > 0$ is any integer 
such that $mu$ lies in the dual lattice $M$ of $N$.
In particular, for the anticanonical region 
$\mathcal{A}$ of $Z$ we have 
$$
\mathcal{A} \cap \sigma 
\ = \ 
\{v \in \sigma; \ \langle u_\sigma, v\rangle \ge -1\}.
$$
This turns $\mathcal{A}$ into a polyhedral complex
and properties of $Z$ being log terminal, canonical 
or terminal become questions on boundedness and
behaviour with respect to lattice points of this 
polyhedral complex.
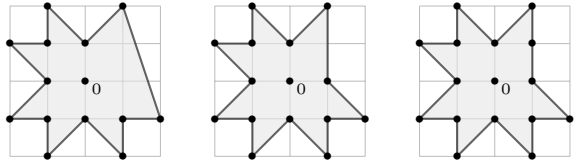
\begin{center}
\begin{tabular}{ccc}
\begin{tikzpicture}[scale=0.5]
\draw [color=gray!50]  [step=1] (-2,-2) grid (2,2);

\draw[thick, fill=gray!20, draw=black, opacity=0.6] 
(0,1) -- (1,2) -- (2,-1) --
(1,-1) -- (1,-2) -- (0,-1) -- 
(-1,-2) -- (-1,-1) -- (-2,-1) --
(-1,0) -- (-2,1) -- (-1,1) --
(-1,2) -- (0,1){};

\foreach \Point in {(-1,0), (0,1), (0,-1), (1,-1), (-1,1), (-1,-1), 
(1,2), (1,-2), (-2,1),
(-1,2), (2,-1), (-1,-2), (-2,-1)}{
    \node at \Point {\tiny \textbullet};}
    
    \node at (0,0) {\tiny \textbullet};
    \node at (0.3 , -0.2) {\tiny $0$};
\end{tikzpicture}
&
\begin{tikzpicture}[scale=0.5]
\draw [color=gray!50]  [step=1] (-2,-2) grid (2,2);

\draw[thick, fill=gray!20, draw=black, opacity=0.6] 
(0,1) -- (1,2)-- (1,0) -- (2,-1) --
(1,-1) -- (1,-2) -- (0,-1) -- 
(-1,-2) -- (-1,-1) -- (-2,-1) --
(-1,0) -- (-2,1) -- (-1,1) --
(-1,2) -- (0,1){};

\foreach \Point in {(-1,0), (0,1),(1,0), (0,-1),  (1,-1), (-1,1), (-1,-1), 
(1,2), (1,-2), (-2,1),
(-1,2), (2,-1), (-1,-2), (-2,-1)}{
    \node at \Point {\tiny \textbullet};}
    
    \node at (0,0) {\tiny \textbullet};
    \node at (0.3 , -0.2) {\tiny $0$};
\end{tikzpicture}
&
\begin{tikzpicture}[scale=0.5]
\draw [color=gray!50]  [step=1] (-2,-2) grid (2,2);

\draw[thick, fill=gray!20,opacity=0.6, draw=black] 
(0,1) -- (1,2) -- (1,1) -- (1,0) -- (2,-1) --
(1,-1) -- (1,-2) -- (0,-1) -- 
(-1,-2) -- (-1,-1) -- (-2,-1) --
(-1,0) -- (-2,1) -- (-1,1) --
(-1,2) -- (0,1){};

\foreach \Point in {(1,0), (-1,0), (0,1), (0,-1), (1,1), (1,-1), (-1,1), (-1,-1), 
(1,2), (1,-2), (-2,1),
(-1,2), (2,-1), (-1,-2), (-2,-1)}{
    \node at \Point {\tiny \textbullet};}
    
    \node at (0,0) {\tiny \textbullet};
    \node at (0.3 , -0.2) {\tiny $0$};
\end{tikzpicture}
\end{tabular}
\end{center}
In the pictures above, we look at the complete fan $\Sigma$ 
having the bullets different from the origin as its primitive 
ray generators. 
Then the shadowed areas indicate the anticanonical regions 
of a log terminal, a canonical and a terminal 
(hence smooth) projective toric surface $Z$ defined by $\Sigma$. 
Note that the polyhedral complexes drawn above are not convex, which implies that
the corresponding toric variety is not Fano. We will investigate this correlation in more
generality in Corollary \ref{cor:fanoConvex}.
\end{remark}

\begin{definition}
Let $X \subseteq Z$ be as in Definition \ref{def:antiCanRegion} and assume that the anticanonical region $\mathcal{A}$ can be endowed with the structure of a polyhedral complex.
In this situation we refer to the anticanonical region as the \emph{anticanonical complex}
and say that $X \subseteq Z$ 
\emph{admits an anticanonical complex}.
\end{definition}

\begin{remark}\label{rem:charAKK}
Let $X\subseteq Z$ admits an anticanonical complex.
Then Remark \ref{rem:anticanRegionSingTypes} specializes to the following:
\begin{enumerate}
    \item[(i')] $X$ has at most log terminal singularities if and only if $\mathcal{A}$ is bounded. 
\item[(ii')]
$X$ has at most canonical singularities
if and only if $0$ is the only 
lattice point in the relative interior of 
$\mathcal{A}$.
\item[(iii')]
$X$ has at most terminal singularities
if and only if $0$ and the primitive generators 
of the rays of the fan of $Z$ are the only 
lattice points of~$\mathcal{A}$.
\end{enumerate}
\end{remark}

Our main result of this section is a sufficient criterion on when an embedding $X\subseteq Z$ admits an anticanonical complex.
Let $X \subseteq Z$ be adapted. Then the pullback homomorphism on the level of Weil divisors described above induces a pullback homomorphism $\Cl(Z) \rightarrow \Cl(X)$ on the level of divisor class groups. We call the embedding $X \subseteq Z$ \emph{neat} if this homomorphism is an isomorphism.

\begin{proposition}\label{prop:AKKForLocallyToric}
Let $X \subseteq Z$ be a neat embedding
admitting a semi-locally toric weakly tropical resolution and assume there exists a $T$-invariant $\QQ$-Cartier divisor $D$ on $Z$ whose pullback $D|_X$ is a canonical divisor on $X$.
Then $X\subseteq Z$ admits an anticanonical complex.
\end{proposition}

The rest of this section 
is dedicated to the proof of the above result.
Starting with an adapted embedding $X \subseteq Z$
in a first step we explicitely construct a
polyhedral complex out of its weakly tropical resolution.
We then show under which conditions one can read discrepancies of $X$ off this complex.
In the second step we show that under the assumptions
of Proposition \ref{prop:AKKForLocallyToric} this complex
is indeed the anticanonical complex of $X \subseteq Z$.

\begin{remark}\label{rem:commDiagramm}
Let $X \subseteq Z$ be adapted and consider an adapted toric ambient modification $Z'\rightarrow Z$. Then there is a commutative diagram
$$
\xymatrix{
{\WDiv^{T'}(Z')}
\ar[r]
\ar[d]_{\varphi_*}
&
{\WDiv(X')}
\ar[d]^{\varphi_*}
\\
{\WDiv^T(Z)}
\ar[r]
&
{\WDiv(X)},
}
$$
where $T' \subseteq Z'$ is the acting torus of $Z'$,
the horizontal arrows are the pullback homomorphisms
defined above and the $\varphi_*$ are the usual birational transforms of Weil divisors via $\varphi$, i.e.\ for any prime divisor $D$ we have $\varphi_{*}(D):= \overline{\varphi(D)}$
if $\codim(\varphi(D)) = 1$ holds, and $0$ otherwise.
\end{remark}

For any toric variety $Z$, we denote by $k_Z$
the toric canonical divisor on $Z$ given as minus 
the sum over all toric prime divisors.
Here comes our main technical tool for the construction of the anticanonical complex:

\begin{definition}
\label{def:toricCanonPhiFamily}
Let $X \subseteq Z$ be adapted and
$\varphi \colon Z' \to Z$ an 
adapted toric ambient modification.
A \emph{toric canonical $\varphi$-family} 
is a family $(U_i,D_i)_{i \in I}$, 
where  the $U_i \subseteq Z'$ are toric open subsets 
covering $Z'$ and the $D_i$ are $T'$-invariant 
Weil divisors on $Z'$ such that for every $i \in I$ 
the following holds:
\begin{enumerate}
\item
$D_i \vert_{X'}$ is a canonical divisor on $X'$,
\item
on $U_i$ we have $D_i = k_{Z'}$,
\item
the $T$-invariant divisor $\varphi_*(D_i)$ is $\QQ$-Cartier.
\end{enumerate}
\end{definition}

\begin{remark}
Let $X \subseteq Z$ be adapted,
$\varphi \colon Z' \to Z$ an adapted 
toric modification and
$(U_i,D_i)_{i \in I}$ a toric canonical 
$\varphi$-family.
Then, by refining, we can achieve 
that $I = \Sigma'$ holds and the
$U_i = Z'_{\sigma'}$ are the affine toric
charts  of $Z'$. 
\end{remark}

Let $u \in M_\QQ$ be a rational character. Then the multiplicity of $\div(\chi^{u})$
along the divisor $D_Z^\varrho$ corresponding to a ray 
$\varrho \in \Sigma$ is given as $\langle u, v_\varrho \rangle$,
where, as usual, $v_\varrho \in \varrho$ denotes 
the primitive lattice vector inside~$\varrho$.

\begin{construction}
\label{constr:A}
Let $X \subseteq Z$ be adapted,
$\varphi \colon Z' \to Z$ an adapted 
weakly tropical resolution and 
$(Z'_{\sigma'},D_{\sigma'})_{\sigma' \in \Sigma'}$
a toric canonical $\varphi$-family.
For every $\sigma' \in \Sigma'$
choose a $\sigma \in \Sigma$ with 
$\sigma' \subseteq \sigma$
and a $u_{\sigma'} \in M_\QQ$ with 
$\varphi_* D_{\sigma'} = \div(\chi^{u_{\sigma'}})$ 
on $Z_\sigma \subseteq Z$.
Set
$$ 
\mathcal{A}
\ := \ 
\bigcup_{\sigma' \in\Sigma'} A_{\sigma '}  ,
\qquad\qquad
A_{\sigma '}
\ := \ 
\sigma' \cap \{v \in N_\QQ; \ \langle u_{\sigma'} , v \rangle \ge -1 \}.
$$
Then $\mathcal{A}$ admits the structure of a polyhedral complex in $N_\QQ$
by defining the cells to be the faces of the polyhedra 
$A_{\sigma '} \subseteq N_\QQ$.
\end{construction}

\begin{remark}\label{rem:Aindepoffam}
In the situation of Construction~\ref{constr:A},
consider a cone $\sigma' \in \Sigma'$,
a ray $\varrho \preceq \sigma'$,
the corresponding toric prime divisor
$D_{Z'}^{\varrho}$ and $D_{X'}^{\varrho} := D_{Z'}^{\varrho} \vert_{X'}$.
Then we have
$$ 
\varphi^*\varphi_* (D_{\sigma'}|_{X'})
=
\varphi^*((\varphi_* D_{\sigma'})|_X)
=
(\varphi^*\varphi_*D_{\sigma'})|_{X'},
$$
where the first equality is due to the commutative diagram given in Remark \ref{rem:commDiagramm} and the second follows by direct calculation in charts. 
In particular, the discrepancy of $D_{X'}^{\varrho}$
with respect to $X$ is given by 
$$ 
\mathrm{discr}_X(D_{X'}^\varrho)
\ = \ 
- 1 - \langle u_{\sigma'}, v_\varrho \rangle,
$$
as the r.h.s. is the multiplicity of 
$D_{\sigma'} - \varphi^*\varphi_* D_{\sigma'}$
along $D_{Z'}^\varrho$ for any $\sigma' \in \Sigma'$ 
with $\varrho \preceq \sigma'$.
In particular, we conclude that the defining inequalities
$u_{\sigma'} \geq -1$ of $\mathcal{A}$ and thus 
the whole set $\mathcal{A}$ 
do not depend on the choice of the toric canonical 
$\varphi$-family.
\end{remark}

\begin{definition}
Let $X \subseteq Z$ be adapted, 
$\varphi \colon Z' \rightarrow Z$
an adapted weakly tropical resolution
and consider a 
proper adapted toric ambient modification $\psi \colon Z'' \to Z'$.
A \emph{toric canonical $\psi$-family over $Z$} 
is a toric canonical $(\varphi \circ \psi)$-family 
$(V_i,C_i)_{i \in I}$ 
such that $V_i = \psi^{-1}(U_i)$ holds with toric open 
subsets $U_i \subseteq Z'$ and $(U_i, \psi_*C_i)_{i \in I}$
is a toric canonical $\varphi$-family.
\end{definition}

\begin{proposition}\label{prop:applyACC}
Situation as in Construction~\ref{constr:A}.
Let $\psi \colon Z'' \to Z'$ be a 
proper adapted toric ambient modification admitting a
toric canonical $\psi$-family over $Z$
and denote by $X'' \subseteq Z''$ the proper transform of $X \subseteq Z$.
Let $\varrho \in \Sigma''$ be a ray
and, provided $\varrho$ intersects the 
boundary of $\mathcal{A}$,
denote by $v_{\varrho}'$ the intersection 
point.
Then the discrepancy $a_{\varrho}$ 
of $X$ with respect to the divisor 
$D_{X''}^\varrho = D_{Z''}^\varrho \vert_{X''}$ 
satisfies
$$ 
a_{\varrho} 
\ = \ 
\frac{||v_{\varrho}||}{||v'_{\varrho}||}-1,
\quad\text{if }
\varrho \not \subseteq \mathcal{A},
\qquad\qquad  
a_{\varrho} \le -1,
\quad \text{if } \varrho \subseteq \mathcal{A}.
$$
\end{proposition}

\begin{proof}
Denote the toric canonical $\psi$-family over 
$Z$ by $(V_i,C_i)_{i \in I}$.
Refining if necessary, we achieve $I = \Sigma'$ 
and $V_{\sigma'}= Z''_{\sigma'} := \psi^{-1}(Z'_{\sigma'})$.
By Remark~\ref{rem:Aindepoffam},
we may assume $D_{\sigma'} = \psi_* C_{\sigma'}$ 
for constructing the polyhedral complex
$\mathcal{A}$ according to~\ref{constr:A}.
Now choose $\sigma' \in \Sigma'$ and 
$\sigma \in \Sigma$ with 
$\varrho \subseteq \sigma' \subseteq \sigma$.
Moreover, let $u_{\sigma'} \in M_\QQ$ with 
$\varphi_* D_{\sigma'} = \div(\chi^{u_{\sigma'}})$ 
on $Z_\sigma \subseteq Z$.
Set $\pi := \varphi \circ \psi$.
Then, on $Z''_{\sigma'}$, we have 
$$
C_{\sigma'} - \pi^*\pi_* C_{\sigma'}
\ = \ 
\sum_{\eta \subseteq \sigma'} -D^\eta_{Z''} - \pi^* \div(\chi^{u_{\sigma'}})
\ = \ 
\sum_{\eta \subseteq \sigma'} (-1 - \langle u_{\sigma'}, v_\eta \rangle)D^\eta_{Z''},
$$
where $\eta$ runs over the rays of $\Sigma''$ that lie in 
the cone $\sigma'$.
Thus, in particular, our~$\varrho$ occurs among the $\eta$.
Now, applying the pullback homomorphism $D \mapsto D \vert_{X''}$
to these identities gives the ramification formula for a 
canonical divisor on $X''$. 
Thus, if $\varrho \not\subseteq \mathcal{A}$ holds,
then we obtain 
$$ 
a_\varrho
\ = \ 
-1 - \langle u_{\sigma'}, v_\varrho \rangle
\ = \ 
-1 - \frac{||v_{\varrho}||}{||v'_{\varrho}||}
\langle u_{\sigma'}, v_\varrho' \rangle
\ = \ 
-1 + \frac{||v_{\varrho}||}{||v'_{\varrho}||},
$$
using $\langle u_{\sigma'}, v_\varrho' \rangle = -1$,
which just rephrases that $v_\varrho'$ lies on the 
bounding hyperplane $u_{\sigma'} = -1$ of 
$\mathcal{A}$.
If $\varrho \subseteq \mathcal{A}$ holds,
then we have $\langle u_{\sigma'}, v \rangle \ge -1$
even for all $v \in \varrho$.
\end{proof}

Let $X \subseteq Z$ be adapted
with adapted weakly tropical resolution 
$Z' \rightarrow Z$. 
The above result shows that 
the polyhedral complex $\mathcal{A}$ as in Construction \ref{constr:A} is the anticanonical complex of $X \subseteq Z$
if every proper birational toric morphism $\psi\colon Z'' \to Z'$ 
is an adapted toric ambient modification and admits a canonical toric 
$\psi$-family over $Z$.
We show that any subvariety $X \subseteq Z$ meeting the assumptions of Proposition \ref{prop:AKKForLocallyToric}
fulfills this condition.

\begin{lemma}
\label{lem:weaklytoricIsAdapted}
Let $X \subseteq Z$ be an adapted embedding admitting a semi-locally toric weakly tropical resolution $\varphi \colon Z' \rightarrow Z$. Then $\varphi$ is an adapted toric ambient modification.
Moreover, if $X \subseteq Z$ is a neat embedding,
then $X' \subseteq Z'$ is~neat. 
\end{lemma}

\begin{proof}
Note that the complement
$X' \setminus (X'\cap Z_0')$
lies in the union of all $T'$-orbits of
$Z'$ of codimension at least two.
As $X' \subseteq Z'$ is weakly tropical,
Tevelev's criterion implies that $X'\cap Z'_0$ is of codimension at least two
in $X'$. Thus we have a well defined pullback homomorphism $\WDiv^{T'}(Z') \rightarrow \WDiv(X')$.

Now let $D_{Z'}^{\varrho}$ be a $T'$-invariant prime divisor on $Z'$ corresponding to a ray $\varrho \in \Sigma'$.
We claim that the pullback of this divisor 
is a prime divisor on $X'$. 
In order to prove this we may restrict $D_{Z'}^{\varrho}$ to the toric chart $Z'_{\varrho} \cong U(\varrho) \times \tilde{\TT}$ 
as the complement of $X'\cap Z_0'$ is of codimension at least two in $X'$.
Note that the restriction $D^{\varrho}_{Z'}|_{Z'_{\varrho}}$ equals the pullback of the $T'$-invariant divisor $D^{\varrho}_{U(\varrho)}$ on $U(\varrho)$ with respect to the projection.
As $U(\varrho)$ and its preimage under the projection $U(\varrho)\times\tilde{\TT}$ are smooth,
the pullback of $D^\varrho_{Z'}$ to $X'$ equals the intersection of $D^\varrho_{U(\varrho)}$ with $X_\varrho$ inside $U(\varrho)$ and thus is a prime divisor.

For the supplement let $X \subseteq Z$ be neat. 
In order to prove that the pullback induces an isomorphism of divisor class groups, we may assume 
 $Z=Z_0$ and $Z' =Z_0'$
due to the adaptedness of $X \subseteq Z$ and $X' \subseteq Z'$.
In particular, we have a proper toric morphism 
$Z' \rightarrow Z$.
Let $E_1, \ldots, E_r$ be the $T'$-invariant prime divisors in the exceptional locus 
of $Z' \rightarrow Z$. As the embedding $X' \subseteq Z'$ is weakly tropical
and adapted we conclude that the prime divisors in the exceptional locus of
$X' \rightarrow X$  are exactly the pullbacks
$E_1|_{X'}, \ldots, E_r|_{X'}$, where we use Tevelev's criterion to show that these are indeed all. 
Note that these divisors generate free 
subgroups of rank $r$ in $\Cl(Z')$ and $\Cl(X')$ respectively. 
Thus we obtain the following commutative diagram with exact rows

\begin{center}
\begin{tikzcd}
0\arrow[r]&
\bigoplus \ZZ \cdot [E_i]\arrow[d,"\cong"']\arrow[r]&
\Cl(Z')\arrow[d]\arrow[r,"\tilde{\varphi}_*"]&
\Cl(Z)\arrow[d,"\cong"]\arrow[r]&
0\\
0\arrow[r]&
\bigoplus \ZZ \cdot[E_i|_{X'}]\arrow[r]&
\Cl(X')\arrow[r,"\varphi_*"']&
\Cl(X)\arrow[r]&
0,
\end{tikzcd}
\end{center}
where the downward arrows are the pullback homomorphisms, and $\tilde{\varphi}_*$ and $\varphi_*$ denote the canonical push forward homomorphisms.
Applying the Five Lemma we
obtain that the pullback homomorphisms
$\Cl(Z') \rightarrow \Cl(X')$ 
induced by the embedding
$X' \subseteq Z'$
is an isomorphism. 
\end{proof}

\begin{lemma}
\label{lem:ClIsom}
Let $X \subseteq Z$ be a neat, weakly tropical embedding and let 
$U \subseteq Z$ be an open $T$-invariant subvariety. 
Then the pullback homomorphism $\Cl(U) \rightarrow \Cl(X \cap U)$ induced by the embedding $X \cap U \subseteq U$ is an isomorphism. 

\begin{proof}
As $U$ is a $T$-invariant subset of $Z$ and the embedding $X \subseteq Z$ is neat, we have
$$Z \setminus U = D_Z^{\varrho_1} \cup \ldots \cup D_Z^{\varrho_r} \cup B
\quad
\text{and}
\quad
X \setminus (X \cap U) = D_X^{\varrho_1} \cup \ldots \cup D_X^{\varrho_r} \cup (B \cap X)
$$
with $T$-invariant prime divisors $ D_Z^{\varrho_i} \subseteq Z$ and
a $T$-invariant closed subset $B \subseteq Z$ of codimension at least two.
Moreover, as $X \subseteq Z$ is weakly tropical,
we have $\mathrm{codim}_X(B\cap X)\geq 2$.
Thus the following commutative diagram with exact rows
gives the assertion:
\begin{center}
\begin{tikzcd}
\ZZ^r\arrow[d, equal]\arrow[r,"{e_i \mapsto [D_Z^{\varrho_i}]}"]&[20pt]
\Cl(Z)\arrow[d,"\cong"]\arrow[r]&
\Cl(U)\arrow[d]\arrow[r]&
0\\
\ZZ^r \arrow[r,"{e_i \mapsto [D_X^{\varrho_i}]}"']&
\Cl(X)\arrow[r]&
\Cl(X\cap U)\arrow[r]&
0.
\end{tikzcd}
\end{center}
\end{proof}
\end{lemma}

\begin{lemma}
\label{lem:LocallytToricCanPsiFamily}
Let $X \subseteq Z$ be a neat embedding and 
let $\varphi \colon Z' \rightarrow Z$ be 
a semi-locally toric weakly tropical resolution such that there exists a $T$-invariant $\QQ$-Cartier divisor $D$ with $D|_{X}$ a canonical divisor on $X$.
Consider
a proper toric morphism $\psi\colon Z''\rightarrow Z'$ defined by a subdivision of fans $\Sigma''\rightarrow \Sigma'$.
Then the following statements~hold:
\begin{enumerate}
\item 
The embedding $X''\subseteq Z''$ is neat.
\item
$\psi \colon Z'' \rightarrow Z'$ is an adapted toric ambient modification and there exists a toric canonical $\psi$-family
over $Z$.
\end{enumerate}
\end{lemma}
\begin{proof}
We prove (i).
Note that by sufficiently refining the quasifan structure on $\trop(X)$ we may assume 
that the toric morphism
$\varphi \circ \psi \colon Z''\rightarrow Z$ 
arises from a refinement of fans
$\Sigma \sqcap \trop(X) \rightarrow \Sigma$.
It thus defines a weakly tropical resolution 
$Z'' \rightarrow Z$ of $X$ which is semi-locally toric according to Lemma~\ref{lem:resolveLocallyToric}. Applying Lemma~\ref{lem:weaklytoricIsAdapted} we conclude
that $X'' \subseteq Z''$ is a neat embedding.

We come to (ii).
As $\psi\colon Z'' \rightarrow Z'$ is a proper morphism, so is its restriction $X''\rightarrow X'$. In particular, using (i) we obtain that $\psi$ is an adapted toric ambient modification and there exists a $T''$-invariant divisor $D$ on $Z''$ whose pullback $D|_{X''}$ is a canonical divisor on $X''$.
We proceed the proof by constructing a toric canonical $\psi$-family over $Z$.
Let $\sigma' \in \Sigma'$ be any cone. Then, in the notation of Lemma~\ref{lem:resolveLocallyToric}, 
the projection of $X'' \cap Z''_{\sigma'}$ into $V(\sigma')$ is an open subset of $V(\sigma')$. Thus there exists a canonical divisor 
on $X''$ that equals $k_{Z''}|_{X''}$ on $X'' \cap Z''_{\sigma'}$. 
Applying Lemma \ref{lem:ClIsom}
we obtain $[k_{Z''}|_{Z''_{\sigma'}}] = [D|_{Z''_{\sigma'}}] \in \Cl(Z''_{\sigma'})$ and
thus on $Z''_{\sigma'}$
we have
$$
D = k_{Z''} + \mathrm{div}(\chi^u)
$$
with a character $\chi^u$ of $T''$. 
Setting $C_{\sigma'} := D - \mathrm{div}(\chi^{u})$ we obtain a toric canonical $(\varphi \circ \psi)$-family
$(Z''_{\sigma'}, C_{\sigma'})_{\sigma' \in \Sigma'}$.
Due to Lemma \ref{lem:weaklytoricIsAdapted}
the morphism $\varphi$ is an adapted toric ambient modification.
Moreover, by construction $Z''_{\sigma'} = \psi^{-1}(Z'_{\sigma'})$ holds 
and the family $(Z'_{\sigma'}, \psi_*C_{\sigma'})$
is a toric canonical $\varphi$-family.
This proves that
$(Z''_{\sigma'}, C_{\sigma'})_{\sigma' \in \Sigma'}$
is a toric canonical $\psi$-family over $Z$.
\end{proof}

\begin{proof}[Proof of Proposition \ref{prop:AKKForLocallyToric}]
Let $\varphi \colon Z' \rightarrow Z$ be any semi-locally toric weakly tropical resolution. Then $\varphi$ is an adapted toric ambient modification due to Lemma \ref{lem:weaklytoricIsAdapted}.
Now let $\psi \colon Z'' \rightarrow Z'$
be any proper birational toric morphism. Then due to
Lemma~\ref{lem:LocallytToricCanPsiFamily}~(ii) it is
an adapted toric ambient modification and admits a canonical toric $\psi$-family over $Z$. Moreover, Proposition \ref{prop:ambRes} ensures that there exists at least one proper birational toric morphism that induces a resolution of singularities. Now, using Proposition \ref{prop:applyACC} we conclude 
that the support of the polyhedral complex 
$\mathcal{A}$ as constructed in \ref{constr:A} is the anticanonical region of $X \subseteq Z$. This completes the proof. 
\end{proof}

\section{Proof of Theorem \ref{introthm1}}\label{sec:locallyToricExplicit}
In this section we concern ourselves with normal varieties $X$ with finitely generated Cox ring. For these varieties, the Cox ring provides an embedding into a toric variety and applying our results from Section \ref{sec:anticanonicalComplexes}, we obtain a criterion for the existence of anticanonical complexes. Beyond that, if $X$ is endowed with an effective action of an algebraic torus $\TT$, we can deduce the existence of an anticanonical complex from a lower dimensional variety $Y$ which suitably represents the field of rational invariant functions $\CC(X)^{\TT} = \CC(Y)$. In particular, we prove Theorem \ref{introthm1}.

Let $X$ be a semi-projective normal variety with only constant invertible global functions and finitely generated divisor class group $\Cl(X)$.
Then the \emph{Cox ring} of $X$ is the following $\Cl(X)$-graded ring, see \cite[Chap. 1]{ArDeHaLa2015} for the details of the definition:
$$
\mathcal{R}(X):= \bigoplus_{\Cl(X)}\Gamma(X,\mathcal{O}_X(D)).
$$
If the Cox ring of $X$ is finitely generated, this setting leads to a closed embedding of $X$ into a toric variety $Z$,  see \cite[Constr. 2.8]{HaHiWr2019}. 
We will identify $X$ with the embedded variety and refer to $X\subseteq Z$ as an \emph{explicit variety}.

\begin{remark}\label{rem:descripZ}
Let $X\subseteq Z$ be an explicit variety. 
Then the embedding $X\subseteq Z$ is neat, divisor class group, Picard group and Cox ring of $Z$ are given as
$$\Cl(Z) \cong \Cl(X), 
\qquad
\Pic(Z) \cong \Pic(X),
\qquad
\mathcal{R}(Z) = \CC[T_{\varrho}; \ \varrho \in \Sigma^{(1)}],$$
where $\Sigma^{(1)}$ denotes the set of rays of
the fan $\Sigma$ defining the toric variety $Z$. Moreover, the ample divisor classes of $X$ and $Z$ coincide under the isomorphism.
\end{remark}

Using the neat embedding of an explicit variety $X\subseteq Z$ we can directly deduce the following corollary from Proposition \ref{prop:AKKForLocallyToric}:

\begin{corollary}\label{cor:423}
Let $X\subseteq Z$ be a $\QQ$-Gorenstein explicit variety and assume there exists a semi-locally toric weakly tropical resolution.
Then $X \subseteq Z$ admits an anticanonical complex.
\end{corollary}

Now let us furthermore assume that 
the explicit variety $X\subseteq Z$ under consideration is endowed with an effective action of an algebraic torus $\TT$. 
Note that in this situation the embedding can always be rearranged to a $\TT$-equivariant one, see \cite[Thm. 3.10]{HaHiWr2019}. 
More precisely we will work with the concept of an {\em explicit $\TT$-variety} $X\subseteq Z$ in the sense of \cite[Def. 3.8]{HaHiWr2019}.

\begin{construction}\label{constr:MOQ}
Let $X\subseteq Z$ be an explicit $\TT$-variety.
Denote by $N$ the lattice of one-parameter subgroups of the acting torus $T$ on $Z$
and let $\Sigma$ be the defining fan of $Z$.
Denote by $N_{\TT}$ the
sublattice in $N$ corresponding to $\TT\subseteq T$.
Set $N' := N/N_{\TT}$, let $P_1 \colon N \rightarrow N'$ be the projection and $\pi_1 \colon T \rightarrow T'$ the associated homomorphism of tori. Set
$$
\Delta_0:= \left\{P_1(\varrho); \ \varrho \in \Sigma^{(1)}\right\} \cup \left\{0\right\}
$$
and let $Y_0$ be the closure of $\pi_1(X \cap T)$ in the toric variety $Z_{\Delta_0}$ corresponding to the fan $\Delta_0$.
Then $\pi_1$ defines a rational quotient
$X \dashrightarrow Y_0$, i.e.\ a dominant rational map such that $\pi_1^*\CC(Y_0)=\CC(X)^\TT$ holds.
Now let $\Delta$ be any fan having the same rays as $\Delta_0$
and let $Y\subseteq Z_\Delta$ be the closure of $Y_0$.
Then we will call $X\dasharrow Y$ an {\em explicit maximal orbit quotient} for the explicit $\TT$-variety $X\subseteq Z$.
\end{construction}

The rational quotient $X \dasharrow Y$ in Construction \ref{constr:MOQ} is a {\em maximal orbit quotient} in the sense of \cite[Def. 3.12]{HaHiWr2019}. Maximal orbit quotients encode properties of the torus action of an explicit $\TT$-variety and $Y$ is unique up to small birational modifications, i.e. birational morphisms inducing isomorphisms up to codimension two, see \cite{HaHiWr2019} for the details.

\begin{remark}\label{rem:tropX=tropY+}
In the situation of Construction \ref{constr:MOQ} the linear map $P_1\colon N_\QQ\rightarrow N'_\QQ$ maps $\trop(X)$ onto $\trop(Y)$.
In particular, we have 
$\trop(X) = \trop(Y)\oplus\ker(P_1)$, see \cite[Cor. 6.2.15]{MaSt2015}. 
\end{remark}

Starting with an explicit $\TT$-variety $X \subseteq Z$
the explicit maximal orbit quotient as in Construction \ref{constr:MOQ} provides the possibility to check if the variety $X \subseteq Z$ admits an anticanonical complex by studying the lower dimensional variety $Y$:

\begin{theorem}\label{thm:quotKrit}
In the notation of Construction \ref{constr:MOQ}, let $X \subseteq Z$ be a $\QQ$-Gorenstein explicit $\TT$-variety
and $X \dashrightarrow Y$
an explicit maximal orbit quotient such that
\begin{enumerate}
    \item 
    $Y \subseteq Z_{\Delta}$ admits a semi-locally toric weakly tropical resolution,
    \item $P_1$ maps
    $|\Sigma \sqcap \trop(X)|$ into $|\Delta \sqcap \trop(Y)|$.
\end{enumerate}
Then $X \subseteq Z$ admits an anticanonical complex.
\end{theorem}

\begin{proof}
Let $Z_{\Delta'} \rightarrow Z_{\Delta}$ be a semi-locally toric weakly tropical resolution of $Y$. In particular, we have $\Delta' = \Delta \sqcap \trop(Y)$ for a fixed quasifan structure on $\trop(Y)$ and by refining we achieve that $\Delta'$ is a subfan of $\trop(Y)$. 
Consider the quasifan structure
$\left\{
P_1^{-1}(\sigma); \ \sigma \in \trop(Y)
\right\}$ on
$\trop(X)$ and let
$\sigma' \in \Sigma'= \Sigma \sqcap \trop(X)$ be any cone.
We claim that $X'_{\sigma'}$ is semi-locally toric. 

Assume $\delta':= P_1(\sigma') \in \Delta'$ holds.
Note that due to Remark \ref{rem:sublattice}
we can identify $N'$ with a sublattice of $N$ and $Y_0$ 
with its image in $Z$ under the toric morphism defined by the morphism $N' \rightarrow N$.
As $Y'$ is semi-locally toric, there exists a maximal cone $\tau \in \trop(Y)$ and a
decomposition $N' = N'(\delta') \oplus \tilde{N}'$ 
such that the corresponding projection
$\pi_{\delta'}$ maps $Y_{\delta'}$ isomorphically onto and open subset of $U(\delta')$. 
As $P_1^{-1}(\tau)$ is a maximal cone in $\trop(X)$ containing $\sigma'$, any maximal cone of a refined quasifan structure on $\trop(X)$ spans the same 
linear subspace. Therefore, we can choose 
$N(\sigma') = N_{\TT} \oplus N'(\delta')$. In particular, choosing the decomposition
$N = N(\sigma')  \oplus \tilde{N}'$
we obtain a commutative diagram
\begin{center}
  \begin{tikzcd}
    Z_{\sigma'}\arrow[d,"\pi_1",xshift=-14mm]
    \quad \cong \quad
    U(\sigma')\times\tilde{\TT}
    \arrow[d, "\psi\times\id", xshift=8mm]
    \arrow[r,"\pi_{\sigma'}"]&
    U(\sigma')\arrow[d, "\psi"]
    \\
    Z_{\delta'}
    \quad \cong\quad
    U(\delta')\times\tilde{\TT}\arrow[r,"\pi_{\delta'}"]
    & 
    U(\delta'),
    \end{tikzcd}
\end{center}
where $\psi$ is the morphism of affine toric varieties 
arising via the projection of lattices  $N(\sigma')\rightarrow N'(\delta')$
mapping $\sigma'$ onto $\delta'$.
As $\pi_{\delta'}$ maps $Y_{\delta'}$ isomorphically onto 
its image we conclude that the projection
$\pi_{\sigma'}$ maps $\pi_1^{-1}(Y_{\delta'})$ isomorphically onto the open subset $\psi^{-1}(\pi_{\delta'}(Y_{\delta'}))
\subseteq U(\sigma')$. We claim that $X'_{\sigma'}$ equals $\pi_1^{-1}(Y_{\delta'})$:
As $\psi^{-1}(\pi_{\delta'}(Y_{\delta'}))$ is irreducible, so is $\pi_1^{-1}(Y_{\delta'})$.
Thus $X'_{\sigma'} \subseteq \pi_1^{-1}(Y_{\delta'})$ is a closed irreducible subvariety of the same dimension and thus equality holds.

In order to conclude the proof 
it is only left to show that for any $\sigma' \in \Sigma'$
we can achieve $P_1(\sigma') \in \Delta'$ by
sufficiently refining the quasifan structure on $\trop(Y)$.
By construction of $\Sigma'$ we have 
$P_1(\sigma')\subseteq \delta$ for some $\delta \in \Delta'$.
Consider any complete fan $\Delta^c$ with
$P_1(\sigma') \in \Delta^c$. 
Then $\trop(Y) \sqcap \Delta^c$ defines a refined fan structure on $\trop(Y)$ that contains $P_1(\sigma')$
and we set $\Delta'':=\trop(Y) \sqcap \Delta^c \sqcap \Delta'$. Now using Lemma \ref{lem:resolveLocallyToric}
we conclude that the proper transform
$Y''$ with respect to the morphism
$Z_{\Delta''} \rightarrow Z_{\Delta'}$
corresponding to the refinement $\Delta'' \rightarrow \Delta'$
is semi-locally toric as $Y'$ is so. 
Thus by the above considerations we obtain that $X'_{\sigma'}$ is semi-locally toric
and the assertion follows with Corollary 
\ref{cor:423}
\end{proof}

The proof of Theorem \ref{thm:quotKrit} provides indeed the following explicit way to construct a semi-locally toric weakly tropical resolution.

\begin{remark}
\label{rem:LocallyToricXOutOfY}
Let $X\subseteq Z$ and $Y\subseteq Z_\Delta$  be as in Theorem \ref{thm:quotKrit} and let $Z_{\Delta'}\rightarrow Z_\Delta$
be the semi-locally toric weakly tropical resolution
of $Y$.
Fix any quasifan structure on $\trop(Y)$ having $\Delta'$ as a subfan and endow $\trop(X)$ with the quasifan structure defined by the cones $P_1^{-1}(\tau)$ with $\tau\in\trop(Y)$. Then the 
refinement of fans $\trop(X)\sqcap\Sigma\rightarrow \Sigma$
defines a semi-locally toric weakly tropical resolution $Z'\rightarrow Z$ of $X$.
\end{remark}

\begin{corollary}\label{cor:introCor5.7}
Let $X \subseteq Z$ be a $\QQ$-Gorenstein explicit $\TT$-variety
and let $X \dashrightarrow Y$ be an explicit maximal orbit quotient, such that $Y$ is complete.
Then $X \subseteq Z$ admits an anticanonical complex if $Y \subseteq Z_\Delta$ admits a semi-locally toric weakly tropical resolution.
\end{corollary}

\begin{proof}
By construction $Y$ is the closure of $Y_0$ in 
a toric variety $Z_\Delta$, where $\Delta$ contains $\Delta_0$ as a subfan.
Moreover, as $Y$ is complete, the support of the defining fan $\Delta$ contains $|\trop(Y)|$. As $P_1$ maps $|\trop(X)|$ into $|\trop(Y)|$, and $Y$ admits by assumption a semi-locally toric weakly tropical resolution, we meet the conditions of Theorem \ref{thm:quotKrit} and the assertion follows. 
\end{proof}

Note that in the case that $X \subseteq Z$ is a complete 
explicit $\TT$-variety we have $|\trop(X)| \subseteq |\Sigma|$. In particular, in this situation every variety $Y$ fulfilling the conditions of Theorem \ref{thm:quotKrit} has to be complete as this is equivalent to $|\trop(Y)| \subseteq |\Delta|$.

\begin{proof}[Proof of Theorem \ref{introthm1}]
We are in the situation of Corollary \ref{cor:introCor5.7}.
Therefore, $X\subseteq Z$ admits an anticanonical complex and Remark \ref{rem:charAKK}
gives the characterizations of the singularity types.
\end{proof}

\section{General arrangement varieties}\label{section:generalArrangementVarieties}
\noindent
In this section we treat the example class of general arrangement varieties as introduced in \cite[Sec. 6]{HaHiWr2019}.
After recalling the basic facts about these $\TT$-varieties we construct an explicit maximal orbit quotient in order to use our Theorem \ref{thm:quotKrit} to prove that they admit anticanonical complexes.

Let us recall the construction of graded rings $R(A,P)$ which are defined by a pair of matrices and that turn out to be the Cox rings of general arrangement varieties, see \cite[Constr.  6.3]{HaHiWr2019}.

\begin{construction}\label{constr:R(A,P)}
Fix integers $r \ge c > 0$
and $n_0, \ldots, n_r > 0$ 
as well as $m \ge 0$. 
Set $n := n_0 + \ldots + n_r$.
Define a pair $(A,P)$, where
\begin{enumerate}
\item 
$A=(a_0,\ldots,a_r)$ is a $(c+1) \times (r+1)$ matrix over $\CC$ such that any $c+1$ of its columns 
are linearly independent,
\item
$P$ is an integral $(r+s) \times (n+m)$ matrix
built from tuples of positive integers 
$l_i = (l_{i1},\ldots,l_{in_i})$, 
where $i = 0, \ldots, r$
and a $s\times (n+m)$ matrix $D$
as follows
$$
P
\ := \
\left[
\begin{array}{c}
\begin{array}{ccccccc}
-l_{0} & l_{1} &  & 0 & 0  &  \ldots & 0
\\
\vdots & \vdots & \ddots & \vdots & \vdots &  & \vdots     
\\
-l_{0} & 0 &  & l_{r} & 0  &  \ldots & 0
\end{array}
\\
\hline
\\
D
\end{array}
\right]
,
$$
whereby we require the columns of the matrix $P$ to be pairwise different, primitive 
and generate $\QQ^{r+s}$ as a vectorspace.
\end{enumerate}
Write  $\CC[T_{ij},S_k]$ for the polynomial ring 
in the variables $T_{ij}$, where $i = 0, \ldots, r$, 
$j = 1, \ldots, n_i$, 
and $S_k$, where $k = 1, \ldots, m$.
Every $l_i$ defines a monomial 
$$
T_i^{l_i}
\ := \ 
T_{i1}^{l_{i1}} \cdots T_{in_i}^{l_{in_i}}
\ \in \ 
\CC[T_{ij},S_k].
$$
Moreover, for every $t = 1, \ldots, r-c$, 
we obtain a polynomial $g_t$ by computing 
the following  $(c+2) \times (c+2)$ 
determinant 
$$ 
g_t 
\ := \
\det
\left[
\begin{array}{cccc}
a_0 & \ldots & a_c & a_{c+t}
\\
T_0^{l_0} & \ldots & T_c^{l_{c}} & T_{c+t}^{l_{c+t}}
\end{array}
\right]
\ \in \ 
\CC[T_{ij},S_k].
$$
Now, let $e_{ij} \in \ZZ^{n}$ 
and $e_k \in \ZZ^{m}$ denote the 
canonical basis vectors
and consider the projection 
$$
Q \colon \ZZ^{n+m} 
\ \to \ 
K := \ZZ^{n+m} / \im(P^*)
$$ 
onto the factor group
by the row lattice of $P$.
Then the 
\emph{$K$-graded $\CC$-algebra
associated with $(A,P)$} 
is defined by
$$ 
R(A,P)
\ := \ 
\CC[T_{ij},S_k] / \bangle{g_1,\ldots,g_{r-c}},
$$
$$
\deg(T_{ij}) :=  Q(e_{ij}),
\qquad
\deg(S_k) :=  Q(e_k).
$$
\end{construction}

The rings $R(A,P)$ are integral normal complete intersections
and their grading allows to use them as Cox rings of certain varieties; see \cite{HaHiWr2019}.

\begin{example}[{Compare 
\cite[Ex. 6.17]{HaHiWr2019}
}] \label{ex:RAP1}
Consider the tuple $(A,P)$ given as
\begin{eqnarray*}
A:=\left[
\begin{array}{cccc}
1&0&0&-1\\
0&1&0&-1\\
0&0&1&-1
\end{array}
\right]
, \qquad
P:=\left[
\begin{array}{ccccc}
-1&-2&2&0&0\\
-1&-2&0&2&0\\
-1&-2&0&0&4\\
-1&-3&1&1&1
\end{array}
\right]
.
\end{eqnarray*}
We obtain a ring
$$R(A,P) = \CC[T_{01},T_{02},T_{11},T_{21},T_{31}] / \bangle{T_{01}T_{02}^2+T_{11}^2+T_{21}^2+T_{31}^4}$$
and the $K = \ZZ^5 / \im(P^*) = \ZZ \times \ZZ/2\ZZ \times \ZZ / 2\ZZ$
grading is given by assigning to every variable $T_{ij}$ the degree $\deg(T_{ij}):=w_{ij}$
where the $w_{ij}$ are the columns of the 
following matrix 
$Q = [w_{01},w_{02},w_{11},w_{21},w_{31}]$
$$
Q:=\left[
\begin{array}{ccccc}
2&1&2&2&1\\
\bar{0}&\bar{0}&\bar{1}&\bar{1}&\bar{0}\\
\bar{0}&\bar{1}&\bar{0}&\bar{1}&\bar{0}
\end{array}
\right].
$$
\end{example}

In the next step we will construct explicit varieties with torus action
having the rings $R(A,P)$ as their Cox ring.

\begin{construction}\label{constr:XAPSigma}
Let $R(A,P)$ be as in Construction \ref{constr:R(A,P)}. The generators $T_{ij},S_k$ of $R(A,P)$ give rise to an embedding
\begin{center}
\begin{tikzcd}
\bar X := \Spec(R(A,P))\arrow[r, hook]
&
\bar Z:=\CC^{n+m}. 
\end{tikzcd}
\end{center}
Fix any fan $\Sigma$ having the columns of $P$ as its primitive ray generators and denote by $Z$ the toric variety with defining fan $\Sigma$.
Define a fan 
$\hat\Sigma:=\{\sigma\preceq\gamma;\  P(\sigma)\in\Sigma\}$, where $\gamma$ denotes the positive orthant and denote by $\hat Z$ the corresponding toric variety. This gives rise to a commutative diagram
\begin{center}
\begin{tikzcd}
\bar X\cap \hat Z \arrow[r, hook]\arrow[d,"p"]
&\hat Z\arrow[d,"p"]
\\
X(A,P,\Sigma)\arrow[r, hook]
&
Z
\end{tikzcd}
\end{center}
where $p$ denotes the toric morphism defined by the linear map $P\colon\CC^{n+m}\rightarrow\CC^{r+s}$
and $X:=X(A,P,\Sigma)$
is the closure of $p(\bar X\cap \TT^{n+m})$
inside $Z$.
By construction the variety $X$ is invariant under the subtorus action $\TT^s\subseteq\TT^{r+s}$ of the acting torus of $Z$.
\end{construction}

The varieties $X:=X(A,P,\Sigma)\subseteq Z$
are normal explicit $\TT^s$-varieties with 
dimension, invertible functions, divisor class group and Cox ring given in terms of their defining data by:
$$
\dim(X) = s + c,
\qquad
\Gamma(X,\mathcal{O}^*) = \CC^*,
\qquad
\Cl(X) = K,
\qquad
\mathcal{R}(X) = R(A,P). 
$$
The torus action of $\TT^s$ is effective and of complexity $c$, i.e.\ the general torus orbit is of codimension $c$.

\begin{example}[{Compare \cite[Ex. 7.9]{HaHiWr2019}
}]\label{ex:RAP2}
Consider the matrices 
$A=[a_0,\ldots,a_3]$ 
and 
$P=[v_{01},v_{02},v_{11},v_{21},v_{31}]$ from Example \ref{ex:RAP1}.
We choose a toric variety $Z$ with defining fan $\Sigma \subseteq \QQ^4$ with maximal cones
$$
\cone(v_{01},v_{11},v_{21},v_{31}),\quad \cone(v_{02},v_{11},v_{21},v_{31}),\quad
\cone(v_{01},v_{02},v_{11}),
$$
$$
\cone(v_{01},v_{02},v_{21}),\qquad
\cone(v_{01},v_{02},v_{31}).
$$
The resulting variety $X(A,P,\Sigma)$
has $R(A,P)$ as its Cox ring and is a Gorenstein Fano variety of dimension three, admits a $\CC^*$-action of complexity two and has Picard number one.
\end{example}

\begin{definition}
We call a variety  $X(A,P,\Sigma)\subseteq Z$ from Construction \ref{constr:XAPSigma}
an {\em explicit general arrangement variety}.
Moreover we call any $\TT$-variety that is equivariantly isomorphic to an explicit general arrangement variety
a \emph{general arrangement variety}.
\end{definition}

The class of general arrangement varieties
comprises i.a. all toric varieties and all
rational $\TT$-varieties of complexity one
with only constant invertible global functions
and $\Gamma(X,\OOO)^\TT = \CC$; see \cite[Rem. 6.2]{HaHiWr2019}.

Let us investigate the explicit maximal orbit quotient $X\dashrightarrow Y_0$ from Construction \ref{constr:MOQ} for general arrangement varieties:

\begin{construction}\label{constr:genArrY}
Let $X:=X(A,P,\Sigma)\subseteq Z$ be an explicit general arrangement variety arising via Construction \ref{constr:XAPSigma}. 
Then $X$ is invariant under the subtorus action of $\TT^s\subseteq\TT^{r+s}$ on $Z$. Using Construction \ref{constr:MOQ}, the projection of lattices $P_1\colon \ZZ^{r+s}\rightarrow\ZZ^r$ with the corresponding projection of tori $\pi_1\colon\TT^{r+s}\rightarrow\TT^r$ give rise to a fan $\Delta_0$ defining a toric variety $Z_{\Delta_0}$ and an explicit variety $Y_0\subseteq Z_{\Delta_0}$:
$$\Delta_0:=\{P_1(\varrho);\ \varrho\in\Sigma^{(1)}\},\qquad Y_0:=\overline{\pi_1(X\cap\TT^{r+s})}\subseteq Z_{\Delta_0}.$$
Moreover, we obtain a commutative diagram of rational quotients:
\begin{center}
\begin{tikzcd}
X\arrow[r,hook]\arrow[d,dashed]
&
Z\arrow[d,dashed]
\\
Y_0\arrow[r,hook]
&
Z_{\Delta_0}.
\end{tikzcd}
\end{center}
As above, let $a_0,\ldots,a_r$ denote the columns of $A$. Then the variety $Y_0\cap \TT^r$ is given as the vanishing set of the linear equations $h_1,\ldots,h_{r-c}$, where
$$ 
h_t
:=
\det
\left[
\begin{array}{ccccc}
a_0 &a_1& \ldots & a_c & a_{c+t}
\\
1 &U_1 &\ldots & U_c & U_{c+t}
\end{array}
\right]
\in 
\CC[U_1^\pm, \ldots, U_r^\pm].
$$
\end{construction}

Note that $\Delta_0$ is the one-skeleton of the defining fan of $\PP_r$. Moreover, the closure of $Y_0\cap\TT^r$ inside $\PP_r$ is a linear subspace $\PP_c\subseteq \PP_r$ and the equations of this embedding are given via the kernel of the matrix $A=(a_0,\ldots,a_r)$.
Note that $X\dashrightarrow \PP_c$ defines  an explicit maximal orbit quotient as in Construction \ref{constr:MOQ}. 
Intersecting $\PP_c$ with the coordinate hyperplanes of $\PP_r$ yields the following collection of hyperplanes building a hyperplane arrangement in general position and explaining the name of these varieties:
$$ 
H_0, \ldots, H_r
\ \subseteq \ 
\PP_c,\qquad
H_i 
\ := \ 
\{z \in \PP_c; \ a_{i0}z_0 + \ldots + a_{ic}z_c = 0\}.
$$

\begin{remark}\label{rem:tropXGenArr}
In the situation of Construction \ref{constr:genArrY} the tropical variety of $Y_0$ is the $c$-skeleton of the fan of $\PP_r$, i.e.\  
$$\trop(Y_0\cap \TT^r)=\Sigma_{\PP^r}^{\leq c}:=\left\{\sigma\in\Sigma_{\PP^r};\  \dim(\sigma)\leq c\right\}.$$
Using Remark \ref{rem:tropX=tropY+} we conclude $|\trop(X)|= |\Sigma^{\leq c}_{\PP_r}|\times\QQ^s$.
In the following, if not specified otherwise, we will always assume $\trop(X)$ to be endowed with the quasifan structure defined by the product $\Sigma_{\PP_r}^{\leq c}\times\QQ^s$.
\end{remark}

\begin{example}\label{ex:RAP3}
Consider the explicit general arrangement variety ${X:=X(A,P,\Sigma) \subseteq Z}$ from Examples \ref{ex:RAP1} and \ref{ex:RAP2}. The $\TT^1$-action on $X$ arises as a subtorus action of $\TT^4$ acting on $Z$.
Using the projection of tori $\TT^{3+1}\rightarrow \TT^3$, we obtain a rational quotient $X\dashrightarrow Y_0$, where 
$$Y_0 \cap \TT^3 \cong \V_{\TT_3}(1+U_1+U_2+U_3) \subseteq  \TT^3 $$
and the tropical variety of $X$ is given as
$\trop(Y_0\cap\TT^3) \times \QQ = \Sigma_{\PP_3}^{\leq 2}  \times \QQ.$
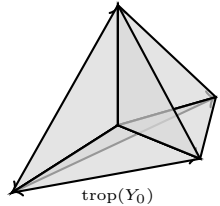
\begin{center}
\tdplotsetmaincoords{70}{110}
\begin{tikzpicture}[tdplot_main_coords]
\tdplotsetrotatedcoords{60}{00}{0}
\draw[thick,tdplot_rotated_coords,->] (0,0,0) -- (1.7,0,0){};
\draw[thick,tdplot_rotated_coords,->] (0,0,0) -- (0,1.7,0){};
\draw[thick,tdplot_rotated_coords,->] (0,0,0) -- (0,0,1.7) {};
\draw[thick,tdplot_rotated_coords,->] (0,0,0) -- (-1,-1,-1) {};
\draw[thick,tdplot_rotated_coords, draw=black, fill=gray!30!,fill opacity=0.7] (0,0,1.7) -- (-1,-1,-1) -- (0,0,0) -- cycle; 
\draw[thick,tdplot_rotated_coords, draw=black, fill=gray!30!,fill opacity=0.7] (0,1.7,0) -- (-1,-1,-1) -- (0,0,0) -- cycle; 
\draw[thick,tdplot_rotated_coords, draw=black, fill=gray!30!,fill opacity=0.7] (1.7,0,0) -- (-1,-1,-1) -- (0,0,0) -- cycle; 
\draw[thick,tdplot_rotated_coords, draw=black, fill=gray!30!,fill opacity=0.7] (0,0,1.7) -- (0,1.7,0) -- (0,0,0) -- cycle; 
\draw[thick,tdplot_rotated_coords, draw=black, fill=gray!30!,fill opacity=0.7] (0,1.7,0) -- (1.7,0,0) -- (0,0,0) -- cycle; 
\draw[thick,tdplot_rotated_coords, draw=black, fill=gray!30!,fill opacity=0.7] (0,0,1.7) -- (1.7,0,0) -- (0,0,0) -- cycle; 
\node at (0,0,-1) {\tiny $\trop(Y_0)$};
\end{tikzpicture}
\end{center}
\end{example}

\begin{theorem}\label{thm:weakTropicalLocallyToric}
Let $X:=X(A,P,\Sigma)\subseteq Z$ be an explicit general arrangement variety. Then the weakly tropical resolution $Z'\rightarrow Z$ of $X$ is semi-locally toric. 
In particular, if $X$ is $\QQ$-Gorenstein, then $X$ admits an anticanonical complex.
\end{theorem}

\begin{lemma}\label{lem:WeakToricGenArr}
Let $A=(a_0,\ldots,a_r)$ be a matrix as in Construction \ref{constr:R(A,P)} and consider the linear subspace $\PP_c\subseteq \PP_r$ defined via the kernel of $A$, i.e.\ the vanishing set of the relations
$f_1, \ldots, f_{r-c}$, where 
$$ 
f_t
:=
\det
\left[
\begin{array}{ccccc}
a_0 &a_1& \ldots & a_c & a_{c+t}
\\
U_0 &U_1 &\ldots & U_c & U_{c+t}
\end{array}
\right]
\in 
\CC[U_0, \ldots, U_r].
$$
Fix the fan structure $\Delta:=\Sigma_{\PP_r}^{\leq c}$ on $\trop(\PP_c)$. 
Then the weakly tropical resolution of $\PP_c \subseteq \PP_r$ is semi-locally toric. 
\end{lemma}

\begin{proof}
As the tropical variety of $\PP_c \subseteq \PP_r$ is
a subfan of the fan of $\PP_r$, Tevelev's criterion implies that the weakly tropical resolution is the identity on $\PP_c$. Therefore, we only have to show that $\PP_c \subseteq Z_{\Delta}$ is semi-locally toric. 
Let $\delta \in \Delta$ be a maximal cone. 
We consider the situation exemplarily for $\delta = \cone(e_{1}, \ldots, e_c)$, i.e.\, we have
$Y_{\delta} \subseteq \CC^c \times (\CC^*)^{r-c}$.
By construction
$Y\cap \TT^r$ is given as the vanishing set of the linear equations $h_1,\ldots,h_{r-c}$, where
$$ 
h_t
:=
\det
\left[
\begin{array}{ccccc}
a_0 &a_1& \ldots & a_c & a_{c+t}
\\
1 &U_1 &\ldots & U_c & U_{c+t}
\end{array}
\right]
\in 
\CC[U_1^\pm, \ldots, U_r^\pm].
$$
Therefore, any point in $Y_{\delta}$ can be written as
$(t, \eta_{1}(t), \ldots, \eta_{r-c}(t))$
where $t \in \CC^c$ and the $\eta_i$ are affine linear forms. 
This implies that $\PP_c\subseteq\PP_r$ is semi-locally toric as the projection $\pi_{\delta}$ maps $Y_{\delta}$ isomorphically onto the following open subset of $\CC^c$:
$$\pi_{\delta}(Y_{\delta}) = \left\{t \in \CC^c; \ \eta_i(t) \neq 0 \text{ for } 1 \leq i \leq r-c\right\}.$$
\end{proof}

\begin{proof}[Proof of Theorem \ref{thm:weakTropicalLocallyToric}]
We show that an explicit general arrangement variety $X(A,P,\Sigma) \subseteq Z$ fulfills the conditions of Theorem \ref{thm:quotKrit}.
Due to Construction \ref{constr:genArrY}
the fan ${\Delta_0}$ is a subfan of the defining fan of $\PP_r$. In particular, in the notation of Theorem
\ref{thm:quotKrit} we may choose $\Delta:= \Sigma_{\PP_r}$ and obtain $Y = \PP_c$ as the closure of $Y_0$ in $\PP_r$. As this embedding is defined via the kernel of $A$ we can apply
Lemma~\ref{lem:WeakToricGenArr} and obtain that
the embedding $\PP_c \subseteq \PP_r$ is weakly tropical and semi-locally toric. Thus it is only left to show that
in this situation we meet condition (ii) of Theorem~\ref{thm:quotKrit}. This follows as $\PP_c$ is complete and thus
$|\Delta \sqcap \trop(Y)| = |\trop(Y)|$ holds.
\end{proof}

\section{Explicit description of anticanonical complexes for general arrangement varieties}\label{section:structuralResultsForGeneralArrangementVarieties}
In this section we give an explicit description
of anticanonical complexes of general arrangement
varieties $X:=X(A,P,\Sigma)\subseteq Z$, see Proposition \ref{prop:AOfGArr} and Corollary \ref{cor:AOfGArr}.
After fixing a quasifan structure on $\trop(X)$ we investigate
the fan of the weakly tropical resolution ${\trop(X)\sqcap\Sigma}$,
see Proposition \ref{prop:raysWeaklyTrop}.
In particular, we obtain in Corollary \ref{cor:weakTropGenArrStays} that the weakly tropical resolution of 
an explicit general arrangement variety is again an explicit general arrangement variety.
Applying our description of the anticanonical complexes and our characterization of 
the several singularity types, 
we prove Theorem \ref{thm:3}, 
which gives first bounding conditions
on the exponents $l_{ij}$ occurring in the defining relations of the Cox ring of $X$.
Specializing to torus actions of complexity two, we obtain concrete bounds for the exponents in the defining equations in the log terminal case as stated in Corollary \ref{introcor2}.

In this section let $X:=X(A,P,\Sigma) \subseteq Z$ always be an explicit general arrangement variety of complexity $c$ and let $\trop(X)\subseteq\QQ^{r+s}$ be its tropical variety endowed with the quasifan structure given in Remark~\ref{rem:tropXGenArr}.

\begin{construction}\label{constr:leaves}
Denote by $e_1,\ldots,e_{r+s}$ the canonical basis of $\QQ^{r+s}$ and set $e_0:=-\sum e_i$.
For any subset $I\subseteq\{0,\ldots,r\}$  of $k$ indices we set
$$\lambda_I:=\cone(e_i;\ i\in I) + \lin(e_{r+1},\ldots,
e_{r+s}).
$$ 
If $1\leq k\leq c$ holds, then we have $\lambda_I\in\trop(X)$ and we call 
$\lambda_I$ a {\em $k$-leaf of $\trop(X)$}.
Moreover, the collection of all leafs of $\trop(X)$  determines the {\em lineality space of $\trop(X)$}: 
$$\lambda_\lin:=\bigcap\limits_{\tiny
I\subseteq\{0,\ldots, r\},
|I|\leq c}
 \lambda_I.$$
\end{construction}

\begin{definition}
In the notation of Construction \ref{constr:leaves} we say that
\begin{enumerate}
\item
a cone $\sigma \in \Sigma$ is a 
\emph{leaf cone} if $\sigma \subseteq \lambda_I$ 
holds for a leaf $\lambda_I$ of $\trop(X)$.
\item
a cone $\sigma \in \Sigma$ is called \emph{big} 
if $\sigma \cap \lambda_i^\circ \neq \emptyset$ 
holds for all $1$-leaves $\lambda_i$ of $\trop(X)$.
\end{enumerate}
\end{definition}

Note that any cone $\sigma\in\Sigma$ is either a big or a leaf cone, see \cite[Prop. 5.5]{HaHiWr2019}.
In particular, an explicit general arrangement variety $X(A,P,\Sigma)\subseteq Z$ is weakly tropical if and only if $\Sigma$ consists of leaf cones.

\begin{construction}\label{constr:vtauprime}
Denote by $v_{ij} := P(e_{ij})$ and 
$v_k := P(e_k)$ the columns of~$P$. Consider a pointed cone of the form
$$ 
\sigma
\ = \ 
\cone(v_{0j_0}, \ldots, v_{rj_r})
\ \subseteq \ 
\QQ^{r+s},
$$
that means that $\sigma$ contains exactly one vector
$v_{ij}$ for every $i = 0,\ldots,r$.
We call such a cone $\sigma$ a \emph{$P$-elementary cone}
and associate to it the following numbers
$$
\ell_{\sigma,i} 
\ := \ 
\frac{l_{0j_0} \cdots l_{rj_r}}{l_{ij_i}}
\text{ for } i = 0, \ldots, r,
\qquad
\ell_{\sigma}
\ := \ 
(c-r) l_{0j_0} \cdots l_{rj_r} + \sum_{i=0}^r \ell_{\sigma, i}
$$
Moreover, we set 
$$ 
v_\sigma
\ := \ 
\ell_{\sigma,0} v_{0j_0} + \ldots +  \ell_{\sigma,r} v_{rj_r}
\ \in \ 
\ZZ^{r+s},
\qquad
\varrho_\sigma 
\ := \ 
\QQ_{\ge 0} \cdot v_\sigma
\ \in \ 
\QQ^{r+s},
$$
and denote by $c_\sigma$ the greatest common divisor of the entries of $v_{\sigma}$. 
\end{construction}

Recall that, if $X$ is $\QQ$-Gorenstein with weakly tropical resolution $Z' \rightarrow Z$, then $X'\subseteq Z'$ is semi-locally toric due to Theorem \ref{thm:weakTropicalLocallyToric}. 
Therefore, $X$ admits an anticanonical complex $\mathcal{A}$ as provided by Construction \ref{constr:A},
which is 
locally defined by linear forms 
$u_{\sigma'}\in M_\QQ$.
More precisely we have
$$
A_{\sigma'} 
= 
\mathcal{A}\cap \sigma' 
=  
\sigma' \cap 
\left\{v \in N_\QQ; \ \bangle{u_{\sigma'},v} \geq -1 \right\}.
$$
We call any $u\in M_\QQ$ fulfilling the above equation a \emph{defining linear form} for $A_{\sigma'}$.
In the following we fix the polyhedral complex structure defined by the polyhedra~$A_{\sigma'}$
and call a point $x \in \mathcal{A}$ a \emph{vertex of $\mathcal{A}$} if 
it is a vertex of one of the polyhedra~$A_{\sigma'}$. 

The following proposition gives a description of the linear forms $u_{\sigma'}$ and thus the anticanonical complex $\mathcal{A}$ in terms of the numbers defined above.

\begin{proposition}\label{prop:AOfGArr}
Let $X(A,P, \Sigma)\subseteq Z$ be a $\QQ$-Gorenstein explicit general arrangement variety. Then
any $u \in M_{\QQ}$ is a defining linear form for $A_{\sigma'}$ if and only if it fulfills the following conditions:
$$
\bangle{u_{\sigma'}, v} =
\begin{cases}
-1, & \text{if } v = v_{\varrho}, \text{ where } \varrho \in (\sigma')^{(1)} \cap \Sigma^{(1)}.
\\
- \ell_{\sigma}, &\text{if } v= v_{\sigma},  \text{ where } \sigma \in \Sigma \text{ is a $P$-elementary cone
    with } \varrho_{\sigma} \preceq \sigma'. 
\end{cases}
$$
\end{proposition}

\begin{corollary}\label{cor:AOfGArr}
Let $X:=X(A,P, \Sigma)\subseteq Z$ be a $\QQ$-Gorenstein general arrangement variety. Then the vertices of the anticanonical complex of $X\subseteq Z$ are the origin, the primitive ray generators of
$\Sigma$ and the points
$v_{\sigma}':= \ell_\sigma^{-1}v_\sigma$, where $\sigma \in \Sigma$ is a $P$-elementary cone and $\ell_{\sigma} > 0$ holds. 
Moreover, if $\ell_{\sigma} > 0$ holds for all $P$-elementary cones $\sigma \in \Sigma$, then 
$X$ is log terminal and each polyhedron $A_{\sigma'}$ is a polytope and therefore determined by the above vertices. 
\end{corollary}

\begin{example}
Consider the affine explicit general arrangement variety $X:=X(A,P,\Sigma) \subseteq Z$ where $A$ and $P$ are as follows and $\Sigma$ is defined by the maximal cone $\sigma$:
$$A=\left[\begin{array}{ccc}
1&0&-1\\
0&1&-1
\end{array}\right],\qquad P=\left[\begin{array}{ccc}
-3&4&0\\
-3&0&4\\
1&1&1
\end{array}\right]
,\qquad \sigma=\cone(v_{01},v_{11},v_{21}).$$
Then $\ell_{\sigma} = -8$ holds, we have $v_\sigma' = (0,0,-5)$ holds and the anticanonical complex is not bounded. In particular, $X$ is not log terminal:

\begin{center}
\tdplotsetmaincoords{50}{110}
\begin{tikzpicture}
[tdplot_main_coords,scale=0.15]
\tdplotsetrotatedcoords{60}{0}{0}

\draw[thick,tdplot_rotated_coords,-] (-10,-10,-10) -- (-10,-10,10) -- (0,0,10) -- (0,0,-10) -- (-10,-10,-10){};
\draw[thick,tdplot_rotated_coords,-] (0,0,10) -- (10,0,10) -- (10,0,-10) -- (0,0,-10) -- (0,0,10){};
\draw[thick,tdplot_rotated_coords,-] (0,0,10) -- (0,10,10) -- (0,10,-10) -- (0,0,-10) -- (0,0,10){};

\draw[thick,tdplot_rotated_coords,-] (0,0,0) -- (-10,-10,10/3){};
\draw[thick,tdplot_rotated_coords,-] (0,0,0) -- (10,0,2.5){};
\draw[thick,tdplot_rotated_coords,-] (0,0,0) -- (0,10,2.5) {};

\draw[thick,tdplot_rotated_coords,-] (0,0,-5) -- (-45/6,-45/6,10) {};
\draw[thick,tdplot_rotated_coords,-] (0,0,-5) -- (10,0,10) {};
\draw[thick,tdplot_rotated_coords,-] (0,0,-5) -- (0,10,10) {};

\draw[thick,tdplot_rotated_coords,fill=black!20, opacity=0.6] (0,0,0) -- (-3,-3,1) -- (-45/6,-45/6,10) -- (0,0,10) -- (0,0,0){};
\draw[thick,tdplot_rotated_coords,fill=black!20, opacity=0.6] (0,0,0) -- (4,0,1) -- (10,0,10) -- (0,0,10) -- (0,0,0){};
\draw[thick,tdplot_rotated_coords,fill=black!20, opacity=0.6] (0,0,0) -- (0,4,1) -- (0,10,10) -- (0,0,10) -- (0,0,0){};

\node[tdplot_rotated_coords] (A) at (-2.5,-2.5,6.5) {$\mathcal{A}$};
\node[tdplot_rotated_coords] (vsigma) at (-1,-1,-7) {$v_\sigma'$};
\node[tdplot_rotated_coords] (rho1) at (-12,-12,10/3) {$\varrho_{01}$};
\node[tdplot_rotated_coords] (rho2) at (15,0,4) {$\varrho_{11}$};
\node[tdplot_rotated_coords] (rho3) at (0,15,2.5) {$\varrho_{21}$};
\end{tikzpicture}
\end{center}
\end{example}

\begin{example}\label{ex:RAP4}
Consider the variety $X:=X(A,P,\Sigma)$ from 
Examples \ref{ex:RAP1}, \ref{ex:RAP2} and \ref{ex:RAP3}. 
We use Proposition \ref{prop:AOfGArr} and Corollary \ref{cor:AOfGArr} to compute the anticanonical complex $\mathcal{A}$ of $X$:
Its vertices are given by the columns $v_{01},v_{02},v_{11},v_{21},v_{31}$ of $P$ and the points
in the lineality space
$$v_{\mathrm{lin 1}}=[0,0,0,1/5],\qquad v_{\mathrm{lin 2}}= [0,0,0,-1/3].$$
The anticanonical complex $\mathcal{A}$ of $X$ has the following $15$ maximal polytopes:
$$
\conv(0,v_{01},v_{02},v_{i1}),\ \conv(0,v_{01},v_{i1},v_{\mathrm{lin 1}}),\ 
\conv(0,v_{02},v_{i1},v_{\mathrm{lin 2}}),\  1\leq i\leq 3,
$$
$$
\conv(0,v_{i1},v_{j1},v_{\mathrm{lin 1}}),\ 
\conv(0,v_{i1},v_{j1},v_{\mathrm{lin 2}}),\ 
1\leq i<j\leq 3.
$$
Besides the origin and the primitive ray generators of $\Sigma$ the anticanonical complex $\mathcal{A}$ of $X$ contains precisely the following lattice points:
$$[0,0,1,0], \quad [1,1,0,1],\quad [1,0,2,1],\quad [0,1,2,1],$$
$$[0,-1,-1,-1],\quad [-1,0,-1,-1],\quad[-1,-1,0,-1],\quad [-1,-1,1,-1].$$
It turns out that {\small$[0,0,0,0]$} is the only lattice point in the relative interior of $\mathcal{A}$ and 
therefore $X$ is a canonical Gorenstein Fano  explicit general arrangement variety of
dimension three, complexity two and Picard number one.
\end{example}

\begin{remark}\label{rem:refinement}
Consider two fans $\Sigma_1$ and $\Sigma_2$ in $\QQ^n$. 
Then the common refinement
$\Sigma_1 \sqcap \Sigma_2$ consists of the cones 
$\sigma_1 \cap \sigma_2$ with $\sigma_i\in\Sigma_i$ for $i=1,2$.
Let $\tau \preceq \sigma_1 \cap \sigma_2$
be any face. Then there exist faces
$\tau_i \preceq \sigma_i$ 
such that 
$\tau = \tau_1 \cap\tau_2$ holds.
\end{remark}

\begin{proposition}
\label{prop:raysWeaklyTrop}
Let $X(A,P, \Sigma)\subseteq Z$ be an explicit general arrangement variety. 
Then the set of rays of $\Sigma \sqcap \trop(X)$ is given by:
$$(\Sigma \sqcap \trop(X))^{(1)} = 
\Sigma^{(1)} \cup \left\{\varrho_\sigma; \ \sigma \in \Sigma \text{ is $P$-elementary}\right\}.
$$ 
\end{proposition}

\begin{lemma}\label{lem:cite}
Let $\sigma \in \Sigma$ be a big cone.
\begin{enumerate}
    \item If $\sigma_1\subseteq \sigma$ is a $P$-elementary cone, then
    $\sigma_1$ is simplicial, we have $v_{\sigma_1} \in \sigma_1^\circ$
    and $\varrho_{\sigma_1} = \sigma_1 \cap \lambda_\mathrm{lin}$ holds.
    \item If $\varrho_{\sigma_1} = \varrho_{\sigma_2}$
    holds for any two $P$-elementary cones $\sigma_1, \sigma_2 \subseteq \sigma$,
    then $\sigma$ is $P$-elementary. In particular, we have $\sigma_1 = \sigma_2 = \sigma.$
\end{enumerate}
\end{lemma}
\begin{proof}
As the definition of a $P$-elementary cone does just depend on the special structure of
the matrix $P$, these statements can be deduced from the proof of 
\cite[Prop. 3.8 (iii),(iv)]{ArBrHaWr2018}.
\end{proof}

\begin{lemma}\label{lem:raysBigCones}
Let $\sigma \in \Sigma$ be a big cone, $\tau \in \trop(X)$ and let $\varrho \in (\sigma \cap \tau)^{(1)}$ be any ray.
Then one of the following statements hold:
\begin{enumerate}
    \item We have $\varrho \in \sigma^{(1)}$.
    \item We have $\varrho = \varrho_{\sigma_1}$, where $\sigma_1 \preceq \sigma$ is a $P$-elementary face.
    In particular, $\varrho \subseteq \lambda_{\mathrm{lin}}$ holds.
\end{enumerate}
\end{lemma}
\begin{proof}
Due to Remark \ref{rem:refinement} there exists $\sigma_{\varrho} \preceq \sigma$
and $\tau_{\varrho} \preceq \tau$ such that $\sigma_{\varrho} \cap \tau_{\varrho} = \varrho$ holds and we may assume these cones to be minimal with this property.
We distinguish between the following two cases.

\vspace{2pt}
\noindent
\emph{Case 1:} We have $\tau_{\varrho} = \lambda_{\mathrm{lin}}$, i.e.\ $\varrho = \sigma_{\varrho} \cap \lambda_{\mathrm{lin}}$.
If we have $\sigma_{\varrho}\subseteq \lambda_{\mathrm{lin}}$, then $\varrho = \sigma_{\varrho} \in \sigma^{(1)}$ holds. 
So, assume not. Then with $\sigma_{\varrho}^\circ \cap \lambda_{\mathrm{lin}} \neq \emptyset$, we conclude that $\sigma_{\varrho}$ is big and there exists a $P$-elementary cone $\sigma_1 \subseteq \sigma_{\varrho}$. We obtain
$$\varrho_{\sigma_1} = \sigma_1 \cap \lambda_{\mathrm{lin}} \subseteq \sigma_{\varrho} \cap \lambda_{\mathrm{lin}} = \varrho
$$
and therefore $\varrho = \varrho_{\sigma_1}$. As this does not depend on the choice of the $P$-elementary cone $\sigma_1$ we conclude that $\sigma_{\varrho}$ is $P$-elementary due to Lemma~\ref{lem:cite}~(ii).

\vspace{2pt}
\noindent
\emph{Case 2:} We have $\varrho = \sigma_{\varrho} \cap \tau_{\varrho}$ with $\varrho \subseteq \tau_{\varrho}^{\circ}$ and
$\tau_{\varrho} \neq \lambda_{\mathrm{lin}}$. 
Assume $\sigma_{\varrho} \subseteq \lambda_{\varrho}$ holds. Then
$\varrho = \sigma_{\varrho} \in \sigma^{(1)}$ holds.
So assume $\sigma_{\varrho} \not \subseteq \lambda_{\varrho}$.
If $\sigma_{\varrho}$ is a leave cone, i.e.\ $\sigma \subseteq \lambda_I \in \trop(X)$ holds, then due to minimality of $\tau_\varrho$ we have $\tau \preceq \lambda_I$. We conclude $\varrho \in (\sigma_{\varrho} \cap \lambda_{I})^{(1)} = \sigma^{(1)}$. So assume $\sigma_{\varrho}$ is a big cone. 
In this case there exists a $P$-elementary cone $\sigma_1 \subseteq \sigma_{\varrho}$ with
$$
\varrho_{\sigma_1} = \sigma_1 \cap \lambda_{\mathrm{lin}} \subseteq \sigma_{\varrho} \cap \tau_{\varrho} = \varrho
$$
and therefore $\varrho = \varrho_{\sigma_1}$. As this does not depend on the choice of the $P$-elementary cone $\sigma_1$ we conclude that $\sigma_{\varrho}$ is $P$-elementary due to Lemma~\ref{lem:cite}~(ii).
\end{proof}

\begin{proof}[Proof of Proposition \ref{prop:raysWeaklyTrop}]
We show "$\subseteq$".
Let $\varrho$ be any ray of $\Sigma \sqcap \trop(X)$.
Then due to Remark \ref{rem:refinement}
we have
$\varrho = \sigma \cap \tau$ with minimal cones
$\sigma \in \Sigma$ and $\tau \in \trop(X)$.
Assume $\sigma$ is a leaf cone, i.e.\ 
$\sigma \subseteq \lambda_I \in \trop(X)$ holds. Then due to minimality of $\tau$ we have
$\tau \preceq \lambda_I$. We conclude
$\varrho \in (\sigma \cap \lambda_I)^{(1)} = \sigma^{(1)}$. 
If $\sigma$ is a big cone, Lemma \ref{lem:raysBigCones} gives the assertion.

We prove "$\supseteq$". Due to construction, the rays of $\Sigma$ are supported on the tropical variety. Thus it is only left to show that $\varrho_{\sigma}$ is a ray of
$\Sigma \sqcap \trop(X)$ for a $P$-elementary cone $\sigma \in \Sigma$.
This follows using Lemma \ref{lem:cite} (ii).
\end{proof}

As a consequence of Proposition \ref{prop:raysWeaklyTrop} we obtain the following corollary:

\begin{corollary}\label{cor:weakTropGenArrStays}
The weakly tropical resolution of an explicit  general arrangement variety is again an explicit general arrangement variety.
\end{corollary}

\begin{proof}
Let $X:=X(A,P,\Sigma)\subseteq Z$ be an explicit general arrangement variety and consider its weakly tropical resolution $Z'\rightarrow Z$. Due to Proposition \ref{prop:raysWeaklyTrop} the rays of $\Sigma'=\Sigma\sqcap \trop(X)$ which are not rays of $\Sigma$ are contained in the lineality space of $\trop(X)$. In particular, the fans $\Delta_0$ and $\Delta_0'$ as in Construction \ref{constr:MOQ} coincide and therefore the explicit maximal orbit quotients $X\dashrightarrow Y_0$ and $X'\dashrightarrow Y_0'$ coincide up to small birational modifications. 
We conclude that $X'=X(A,P',\Sigma')$ holds, where $A$ is the same matrix as for $X$ and $P'$ contains the primitive ray generators of the fan $\Sigma'$, see \cite[Sec. 6]{HaHiWr2019}.
\end{proof}

Due to \cite[Thm. 6.5]{HaHiWr2019} the Cox ring $R(A,P)$ of an explicit general arrangement variety $X:=X(A,P,\Sigma)\subseteq Z$ is a complete intersection ring.
Therefore, we can apply \cite[Prop. 3.3.3.2]{ArDeHaLa2015} and obtain the canonical class of $X$ 
via the following formula:
$$\KKK_X = - \sum_{\varrho \in \Sigma^{(1)}}\deg(T_\varrho)
+ \sum_{i= 1}^{r-c} \deg(g_i) 
\in\Cl(X)\cong\ZZ^{n+m} / \im(P^*).$$

\begin{proposition}
\label{prop:disc}
Let $X:=X(A,P, \Sigma)\subseteq Z$ be a $\QQ$-Gorenstein explicit
general arrangement variety
with weakly tropical resolution $Z'\rightarrow Z$ and let $\sigma \in \Sigma$ be a $P$-elementary cone. Then the following statements hold:
\begin{enumerate}
\item
The discrepancy along the prime divisor of $X'\subseteq Z'$ 
corresponding to $\varrho_\sigma$ equals 
$c_{\sigma}^{-1}\ell_{\sigma}-1$.
\item
The ray $\varrho_\sigma$ is not contained in
the anticanonical complex $\mathcal{A}$, 
if and only if $\ell_\sigma > 0$ holds; 
in this case, $\varrho_\sigma$ leaves $\mathcal{A}$ at
$v_\sigma' = \ell_{\sigma}^{-1} v_\sigma$.
\end{enumerate}
\end{proposition}
\begin{proof}
We prove (i).
Due to Theorem \ref{thm:weakTropicalLocallyToric}
the variety $X \subseteq Z$ admits a semi-locally toric weakly tropical resolution. Applying Lemma \ref{lem:LocallytToricCanPsiFamily} we conclude that there exists a toric canonical $\varphi$-family. Therefore, explicitly constructing a pair 
$(Z'_{\varrho_{\sigma}}, D_{\varrho_\sigma})$ as in Definition
\ref{def:toricCanonPhiFamily} we can 
use Remark \ref{rem:Aindepoffam}
to calculate the discrepancy along 
$D_{X'}^{\varrho_\sigma}$:

Consider the ray $\varrho_{\sigma} \in \Sigma'$. Then 
$\varrho_\sigma \subseteq \lambda_I$ holds
for every maximal leaf of 
$\trop(X)$. In particular, we may choose 
$I:=\left\{ 1, \ldots, c\right\} $ and consider the divisor
$$D_{\varrho_\sigma} := \sum_{j=1}^{n_0}
(r-c)l_{0j} D_{\varrho_{0j}} - \sum_{\varrho' \in (\Sigma')^{(1)}}D_{\varrho'}.$$
Then, as $X'\subseteq Z'$ is an explicit general arrangement variety due to Corollary \ref{cor:weakTropGenArrStays}, the pullback $D_{\varrho_\sigma}|_{X'}$ is a canonical divisor on $X'$. Moreover, the push forward $\varphi_*(D_{\varrho_{\sigma}})$ is $\QQ$-Cartier and by construction we have $D_{\varrho_\sigma} = k_{Z'}$ on $Z'_{\varrho_{\sigma}}$.
In particular, we have constructed a tuple $(Z'_{\varrho_{\sigma}}, D_{\varrho_\sigma})$ as claimed.
Now, let $u \in \QQ^{r+s}$ be an element such that $\mathrm{div}(\chi^{u}) = \varphi_*(D_{\varrho_\sigma})$ holds on $Z_{\sigma}$. Then due to Remark \ref{rem:Aindepoffam} we have
$$
\mathrm{discr}_X(D_{X'}^{\varrho_\sigma}) = -1 - \bangle{u, v_{\varrho_\sigma}}. 
$$
Therefore, using $v_{\sigma} = v_{\varrho_\sigma} \cdot c_{\sigma}$, we obtain the assertion with
$$
\bangle{u, v_{\sigma}} 
=
\bangle{u, \sum_{i=0}^r \ell_{\sigma, i}v_{ij_i}}
= 
\sum_{i=0}^r\ell_{\sigma,i}\bangle{u, v_{ij_i}}
=
- \ell_\sigma.
$$
Using (i) assertion (ii) 
follows from the definition of the anticanonical complex.
\end{proof}

\begin{proof}[Proof of Proposition \ref{prop:AOfGArr}]
Let $\sigma' \in \Sigma'$ be any cone. Then
a linear form $u\in M_\QQ$
is defining for $A_{\sigma'}$ if and only if
for all rays $\varrho \in \sigma'$ we have
$$\mathrm{disc}_X(D_{X'}^\varrho) = -1 - \bangle{u, v_{\varrho}}.$$
Due to Proposition \ref{prop:raysWeaklyTrop}
the rays of $\sigma'$ are either rays of $\Sigma$, then
$\bangle{u, v_{\varrho}} = - 1$ holds, or they are of the form $\varrho_{\sigma}$ where $\sigma \in \Sigma$
is a $P$-elementary cone.
In this case
the assertion follows from
Proposition \ref{prop:disc} (ii).
\end{proof}

\begin{proof}[Proof of Corollary \ref{cor:AOfGArr}]
By definition the vertices of $\mathcal{A}$ are the vertices of $A_{\sigma'}$ where $\sigma'$ runs over all cones of $\Sigma'$.
In particular, they arise as the intersection of the hyperplane $u_{\sigma'} = -1$ with the rays of $\sigma'$.
Therefore, Proposition \ref{prop:AOfGArr}
gives the assertion. The supplement 
follows using the
characterization of log terminality as given in Remark \ref{rem:charAKK} (i').
\end{proof}

\begin{proof}[Proof of Theorem \ref{thm:3}]
The cone $\sigma\in\Sigma_X$ is by definition $P$-elementary and big.
Thus, the assertion follows via direct calculation from Proposition \ref{prop:disc}.
\end{proof}

\begin{remark}\label{rem:expLogTerm}
Consider a $P$-elementary cone 
$\sigma= \varrho_0+ \dots +\varrho_r\in \Sigma$
defining a log terminal singuarity and 
assume $l_{\varrho_0}\geq\dots \geq l_{\varrho_r}$ holds.
Then the condition in Theorem \ref{thm:3} (i) implies that
$\sum_{i=0}^{c+1} l_{\varrho_{i}}^{-1} > 1$ and $l_{\varrho_{c+2}} = \cdots =l_{\varrho_r}= 1$ holds.
\end{remark}

\begin{proof}[Proof of Corollary \ref{introcor2}]
Using Remark \ref{rem:expLogTerm}
the claim follows 
from Theorem \ref{thm:3}
via a direct calculation 
in the complexity two case.
\end{proof}

\section{An alternative construction}
In this section we consider explicit general arrangement varieties
$X:=X(A,P,\Sigma)\subseteq Z$ with ample anticanonical divisor
and give an alternative description of the anticanonical complex in this setting.
Following the same steps as done in \cite{BeHaHuNi2016}, we explicitly construct a polyhedral complex 
and show in Theorem \ref{thm:compACCs} that it is indeed the anticanonical complex of $X$.
In particular, we make the construction developed in \cite{BeHaHuNi2016} applicable in a broader setting: Besides leaving the Fano case by dropping the condition on $X$ to be projective, Example \ref{ex:RAP4} shows that in general the varieties $X(A,P,\Sigma) \subseteq Z$ are not treatable with the methods developed there.

In this section let $X:=X(A,P,\Sigma)\subseteq Z$ be an explicit general arrangement variety with Cox ring $R(A,P)$ given by generators $T_{ij},S_k$ and relations $g_1,\ldots,g_{r-c}$ as in Construction \ref{constr:R(A,P)}. Moreover let $Z'\rightarrow Z$ be its weakly tropical resolution defined by the fan $\Sigma'=\Sigma\sqcap\trop(X)$ in $\QQ^{r+s}$.

\begin{construction}
\label{constr:AnticanPQ}
Let $\gamma_{n+m}\subseteq\QQ^{n+m}$ be the positive
orthant and let $e_\Sigma\in\ZZ^{n+m}$ be 
any representative of the canonical class $\KKK_Z$ of $Z$.
Define polytopes
$$B(-\KKK_X):=Q^{-1}(-\KKK_X)\cap \gamma_{n+m}\subseteq\QQ^{n+m}
\quad
$$
and
$B:=B(g_1)+\ldots+B(g_{r-c})$
as the Minkowski sum of the Newton polytopes $B(g_i)$ of
the relations $g_i$.
The \emph{anticanonical polyhedron} of $X$ 
is the dual polyhedron $A_X \subseteq \QQ^{r+s}$ 
of the polyhedron
$$
B_X
\ := \ 
(P^*)^{-1}(B(-\mathcal{K}_X) + B - e_{\Sigma}) 
\ \subseteq \ 
\QQ^{r+s}.
$$

\end{construction}

\begin{theorem}
\label{thm:compACCs}
Let $X:=X(A,P,\Sigma)\subseteq Z$ be a $\QQ$-Gorenstein explicit general arrangement variety with ample anticanonical class. Then the anticanonical complex $\mathcal{A}$ of $X$
is the polyhedral complex
$$ 
{\rm faces}(A_X) \sqcap \Sigma \sqcap \trop(X) 
= 
{\rm faces}(A_X) \sqcap \Sigma'.
$$
\end{theorem}

\begin{corollary}\label{cor:fanoConvex}
Let $X:=X(A,P,\Sigma)\subseteq Z$ be a Fano general arrangement variety. Then its anticanonical complex $\mathcal{A}$ is piecewise convex, i.e. 
$$
\mathrm{conv}(|\mathcal{A}|) \cap |\trop(X)|
= 
|\mathcal{A}|.
$$
\end{corollary}

\begin{example}\label{ex:RAP5}
Consider the variety $X := X(A,P, \Sigma) \subseteq Z$ from Examples \ref{ex:RAP1}, \ref{ex:RAP2}, \ref{ex:RAP3} and \ref{ex:RAP4}. 
By construction, $X$ is a Fano general arrangement variety and its anticanonical complex is given as
$$
|\mathcal{A}| = \mathrm{conv}(v_{01},v_{02},v_{11},v_{21}, v_{31}, v_{\mathrm{lin}1}, v_{\mathrm{lin}2}) \cap |\trop(X)|.
$$
\end{example}

\begin{lemma}\label{lem:tropSubfan}
Let $X(A,P,\Sigma)\subseteq Z$ be an explicit general arrangement variety with $A=(a_0,\ldots,a_r)$ and consider the linear subspace $\PP_c\subseteq \PP_r$ defined via the kernel of $A$, i.e.\ the vanishing set of the relations
$f_1, \ldots, f_{r-c}$, where 
$$ 
f_t
:=
\det
\left[
\begin{array}{ccccc}
a_0 &a_1& \ldots & a_c & a_{c+t}
\\
U_0 &U_1 &\ldots & U_c & U_{c+t}
\end{array}
\right]
\in 
\CC[U_0, \ldots, U_r].
$$
Then $\trop(\PP_c \cap \TT^r)=\Sigma_{\PP_r}^{\leq c}$ is a subfan of the normal fan of 
$$
\tilde{B}:=B(h_1) + \dots + B(h_{r-c}) \subseteq \QQ^r,
\quad 
\text{ with} 
\quad
h_i  := f_i(1,U_1, \dots, U_r).
$$
In particular, the tropical variety
$\trop(X)$ is a subfan of the normal quasifan of $\tilde{B}$ considered as a 
a polytope in $\QQ^{r+s}$.
\end{lemma}

\begin{proof}
Let $e_1, \ldots, e_r$ denote the standard basis vectors of $\QQ^r$ and set $e_0 := - \sum e_i$.
As $\trop(\PP_c\cap \TT^r)$ is by definition a refinement of a subfan of $\mathcal{N}(\tilde{B})$, it suffices to show that $\cone(e_k)$ is a ray of $\mathcal{N}(\tilde B)$ for every $k = 0, \ldots, r$. For this set 
$$
J_t := \left\{j; \ U_j \text{ is a monomial of } f_t\right\} = \left\{0, \ldots, c, c+t\right\}.
$$
Then the lineality space of $\mathcal{N}(B(h_t))$, i.e. the maximal linear subspace contained in $\mathcal{N}(B(h_t))$,
is
$$
\sigma_t^{\mathrm{lin}} := \mathrm{lin}(e_j; \ j \in \left\{0, \ldots, r\right\} \setminus J_t).
$$
Now let $k \leq c$. Then $\cone(e_k) \times \sigma_t^{\mathrm{lin}} \in \mathcal{N}(B(h_t))$ holds for every $t= 1, \ldots, r-c$ and we claim
$$
\cone(e_k) = \bigcap_{t=1}^{r-c}\left(\cone(e_k) \times \sigma_t^{\mathrm{lin}}\right) =: \sigma \in  \mathcal{N}( \tilde B),
$$
i.e.\ we have to show the inclusion "$\supseteq$".
Let $a \in \sigma$ be any point. Then for every $t = 1, \ldots, r-c$ we have a description
$$
a = \sum_{j> c, \ j \neq t} a_{tj} e_j + b_{tk} e_k, 
\quad
\text{with}
\quad
a_{tj } \in \QQ, \ b_{tj} \in \QQ_{\geq 0}.
$$
Usig that $\left\{e_k, e_c, \ldots, e_r\right\}$ are linearly independant, 
we conclude $a_{tj}= 0$ for all $t,j$ 
and therefore $a \in \cone(e_k)$. 
We come to the case $k > c$. Here we have
$\cone(e_k) \times \sigma_k^{\mathrm{lin}} \in \mathcal{N}(B(h_k))$ and we claim
$$
\cone(e_k) = \left(\cone(e_k) \times \sigma_k^{\mathrm{lin}}\right) \cap \bigcap_{t=1, t \neq k}^{r-c} \sigma_t^{\mathrm{lin}} \in  \mathcal{N}( \tilde B).
$$
Analogously to the first case this can be verified by a direct calculation.
\end{proof}

\begin{lemma}\label{lem:diagramm}
Let $B \subseteq \QQ^{m}$ be any polyhedron and denote by $\mathcal{N}(B)$ its normal quasifan.
Let further $P \colon \QQ^{n} \rightarrow \QQ^m$ be a surjective linear map. 
Then $\mathcal{N}(P^*(B)) = P^{-1}( \mathcal{N}(B))$ holds.
\begin{proof}
Let $B$ be any polyhedron. Then
$\mathcal{N}(B) =\left\{C_F^\vee; \ F \preceq B \text{ face}\right\}$ holds,
where $C_F:= \cone(u-v; \ u \in B, v \in F)$. Thus we have
$$C_F^\vee = \left\{y; \ \bangle{y, u-v} \geq 0 \text{ for all } u \in B, v \in F \right\}.$$
Note that due to injectivity of $P^*$ the faces of $P^*(B)$ are precisely the images $P^*(F)$ of the faces
$F \preceq B$. We conclude
\begin{align*}
C_{P^*(F)}^\vee 
&= \left\{x; \ \bangle{x, P^*(u)-P^*(v)} \geq 0 \text{ for all } u \in B, v \in F\right\}
\\
&=\left\{x; \ \bangle{P(x), u-v} \geq 0 \text{ for all } u \in B, v \in F\right\}
\\
&=P^{-1}(C_F^\vee).
\end{align*}
\end{proof}
\end{lemma}

To a $T$-invariant Weil divisor $D=\sum a_\varrho D_\varrho$ on $Z$ we assign a polyhedron:
$$B_{D} := \left\{u \in M_\QQ; \ \bangle{u, v_\varrho} \geq -a_\varrho \right\}\subseteq M_\QQ.$$

\begin{proposition}
\label{prop:acancompstruct}
Let $X:=X(A,P, \Sigma)\subseteq Z$ be an explicit general arrangement variety with 
ample anticanonical class, fix a $T$-invariant divisor $-k_X$ on $Z$ such that $-k_X|_X$ is an anticanonical divisor on $X$. Let $B_{-k_X}$ denote the polyhedron corresponding to $-k_X$ and let $B \subseteq \QQ^{r+s}$ and $\tilde B \subseteq \QQ^{r+s}$ be as in Construction \ref{constr:AnticanPQ} and Lemma \ref{lem:tropSubfan}.
Then we have the following equalities:
\begin{enumerate}
    \item 
    $P^*(B_{-k_X}) - k_X = B(- \mathcal{K}_X)$.
    \item
    $P^*(\tilde{B}) = B - (r-c)\cdot l_0$, where $l_0$ is identified with $(l_0, 0, \ldots, 0) \in \CC^{n+m}$
\end{enumerate}
In particular, the fan $\Sigma \sqcap \trop(X)$ is a subfan of the normal fan of $B_X$.
\end{proposition}

\begin{proof}
The two equalities follow by direct calculation.
We prove the supplement.
Using Lemma \ref{lem:diagramm} we obtain
$$P^{-1}(\mathcal{N}(B_{-k_X})) = \mathcal{N}(B(-\mathcal{K}_X))
\quad \text{ and }
\quad
P^{-1}(\mathcal{N}(\tilde{B})) = \mathcal{N}(B).$$
As $-k_X$ is ample due to Remark \ref{rem:descripZ}, the fan $\Sigma$ is a subfan of the normal fan $\mathcal{N}(B_{-k_X})$, and Lemma \ref{lem:tropSubfan} shows that 
$\trop(X)$ is a subfan of $\mathcal{N}(\tilde{B})$. It follows that 
$P^{-1}(\Sigma) \sqcap P^{-1}(\trop(X))$
is a subfan of the normal fan of $B(-\mathcal{K}_X) + B$. Projecting the involved fans via $P$ to $\QQ^{r+s}$ gives the assertion.
\end{proof}

\begin{proof}[Proof of Theorem \ref{thm:compACCs}]
Let $\sigma' \in \Sigma'$ be any cone. Then due to Proposition \ref{prop:acancompstruct}
we have $\sigma' \in \mathcal{N}(B_X)$. Fix any maximal cone $\tau \in \mathcal{N}(B_X)$ with $\sigma' \preceq \tau$. Then
we have $\sigma' \preceq \sigma \cap \lambda_I \preceq \tau$,
where $\lambda_I$ is a maximal 
leaf of $\trop(X) = \trop(X')$ and 
$\sigma \in \Sigma$ holds. 
Denote by $u \in B_X$ the vertex corresponding to $\tau$. 
Then we have a decomposition $P^*(u) = \mu + \nu - e_\Sigma$,
with $\mu \in B(-\mathcal{K}_X)$ and $\nu \in B$.
We claim that the family
$$(Z'_{\sigma'}, D_{\sigma'})_{\sigma' \in \Sigma'}, 
\quad
\text{ with }
\quad
D_{\sigma'} = \sum_{\varrho \in (\Sigma')^{(1)}\cap \Sigma^{(1)}}\bangle{\nu, e_\varrho} D_\varrho
- \sum_{\varrho \in (\Sigma')^{(1)}} D_\varrho$$
is a toric canonical $\varphi$-family
and $\div(\chi^u) = \varphi_*(D_{\sigma'})$ holds on $X_\sigma$.

In order to verify the claim
we first show that $\bangle{\nu, e_\varrho} = 0$ holds for all
$\varrho \in (\sigma')^{(1)} \cap \sigma^{(1)}$.
Denote by $\tilde{\nu}$ the vertex of $\tilde{B}$ corresponding to $\nu$.
Then $\tilde{\nu}$ defines the maximal cone
$$\left\{x \in \QQ^{r+s}; \ \bangle{x, u-\tilde{\nu}} \geq 0 \text{ for all } u \in \tilde{B} \right\},$$
which contains $\lambda_I$.
After suitably renumbering we may assume
$I = \left\{1, \ldots, c\right\}$
and thus $\tilde{\nu}_1= \ldots = \tilde{\nu}_c = 0$. With $P^*(\tilde{B}) = B - (r-c) \cdot l_0$ we obtain that 
$$
\bangle{\nu, e_{\varrho}} = 
\bangle{P^*(\tilde{\nu}) - (r-c)l_0, e_{\varrho}}
=
\bangle{\tilde{\nu}, v_{\varrho}}
- 
\bangle{(r-c)l_0, e_{\varrho}}
=
0
$$ 
holds for all $\varrho \in \sigma^{(1)} \cap (\sigma')^{(1)}$.
This shows that 
$(Z'_{\sigma'}, D_{\sigma'})_{\sigma' \in \Sigma'}$ is a toric canonical $\varphi$-family.

It is only left to show that 
$\div(\chi^u) = \varphi_*(D_{\sigma'})$ holds on $X_\sigma$.
We fix a $T$-invariant divisor 
$-k_X = \sum_{\varrho\in \Sigma^{(1)}} a_\varrho D_\varrho$ 
whose pullback $-k_X|_X$ is an anticanonical divisor on $X$ and
denote by $\tilde{\mu}$ 
the vertex in $B_{-k_X}$ 
corresponding 
to $\mu$.
Then, as ${-k_X}|_X$ and therefore $-k_X$ is ample, we have
$\bangle{\tilde{\mu},v_\varrho} = - a_\varrho$ for all rays $\varrho \in \sigma^{(1)}$, see \cite[Prop. 6.2.5]{CoLiSc2011}.
We conclude
$$\bangle{\mu, e_\varrho}
=
\bangle{P^*(\tilde{\mu})- k_X, e_\varrho}
=
\bangle{\tilde{\mu}, v_\varrho} + \bangle{-k_X, e_\varrho} 
= 
0.
$$
Since $\bangle{u, v_\varrho} = \bangle{\mu + \nu - e_\Sigma, e_\varrho}$ holds, this completes the proof.
\end{proof}

\begin{example}\label{ex:RAP6}
Consider the variety $X:= X(A,P,\Sigma) \subseteq Z$ from Examples \ref{ex:RAP1}, \ref{ex:RAP2}, \ref{ex:RAP3}, \ref{ex:RAP4} and \ref{ex:RAP5}. 
As before, we denote the  primitive ray generators of~$\Sigma$ by
$$
[v_{01}, v_{02}, v_{11}, v_{21},v_{31}]
=
\left[
\begin{array}{ccccc}
-1&-2&2&0&0\\
-1&-2&0&2&0\\
-1&-2&0&0&4\\
-1&-3&1&1&1
\end{array}
\right]
$$
and the defining relation of the Cox ring of $X$ by $g = T_{01}T_{02}^2 + T_{11}^2 + T_{21}^2 + T_{31}^4$. 
The common refinement $\Sigma' = \Sigma \sqcap \trop(X)$ is pure of dimension $3$ and 
we have
$$
(\Sigma')^{(1)} = \Sigma^{(1)} \cup 
\left\{\cone(e_4), \ \cone(-e_4)\right\}.$$ Further refining this fan we
obtain a smooth toric variety $Z''$ whose fan has $72$ maximal cones and primitive ray generators given by the columns of the matrix
$$P_2:=
{\tiny\setlength{\arraycolsep}{2pt}
\left[
 \begin{array}{ccccccccccccccccccccccccccccccc} 
 -2&-2&-1&-1&-1&-1&-1&-1&0&0&0&0&0&0&0&0&0&0&0&0&0&0&0&1&1&1&1&1&1&2&2
\\ 
-2&-2&-1&-1&-1&-1&-1&0&-1&0&0&0&0&0&0&0&1&1&1&1&1&2&2&0&0&0&0&0&1&0&0
\\ 
-2&-1&-1&-1&0&0&1&-1&-1&0&0&1&1&2&3&4&0&0&1&1&2&0&1&0&0&1&1&2&0&0&1
\\
-3&-3&-2&-1&-2&-1&-1&-1&-1&-1&1&0&1&1&1&1&0&1&0&1&1&1&1&0&1&0&1&1&1&1&1
\end{array}\right].}
$$
We show that this example leaves the framework in which \cite[Thm. 1.4]{BeHaHuNi2016} can be applied
as $X \subseteq Z$ is not \emph{tropical resolvable} in their sense:
Note that we have
$\cone(v_{02}, v_{31}),\
\cone(v_{21},v_{31}) \in \Sigma'.
$
In particular, any regular refinement of $\Sigma'$ contains the following rays:
$$
\varrho_1:=\cone([-1,-1,1,-1])
\quad
\text{and}
\quad
\varrho_2:=\cone([1,1,0,1]).
$$
Let $P''$ be a matrix whose columns are the primitive generators of the rays of a regular refinement $\Sigma''$ of $\Sigma'$.
Then the \emph{shift} of $g$ with respect to $P''$ and $P$ is the unique polynomial 
$\tilde{g} \in \CC[T_{\varrho}; \ \varrho \in (\Sigma'')^{(1)}]$ without monomial factors satisfying that its push with respect to $P''$ equals the push of $g$ with respect to $P$.
In particular, we have $\tilde{g}=m_1+m_2+m_3+m_4$, where the $m_i$ satisfy
$$
T_{\varrho_1}|m_i \Leftrightarrow 
i = 1,2
\quad
\text{ and}
\quad
T_{\varrho_2}|m_i \Leftrightarrow i = 3,4
$$
In the example $P'' = P_2$ we have
\begin{align*}
m_1&:=T_{{24}}T_{{25}}T_{{26}}T_{{27}}T_{{28}}T_{{29}}{T_{{30}}}^{2}{T_{{31}
}}^{2}T_{{9}}
\\
m_2&:=T_{{17}}T_{{18}}T_{{19}}T_{{20}}T_{{21}}{T_{{22}}}^{2}{T
_{{23}}}^{2}T_{{29}}T_{{8}}
\\
m_3&:={T_{{7}}}^{2}T_{{12}}T_{{13}}{T_{{14}}}^{2
}{T_{{15}}}^{3}{T_{{16}}}^{4}T_{{19}}T_{{20}}{T_{{21}}}^{2}T_{{23}}T_{
{26}}T_{{27}}{T_{{28}}}^{2}T_{{31}}T_{{2}}T_{{5}}T_{{6}}
\\
m_4&:={T_{{1}}}^{2}
{T_{{2}}}^{2}T_{{3}}T_{{4}}T_{{5}}T_{{6}}T_{{7}}T_{{8}}T_{{9}},
\end{align*}
where $T_{7} = T_{\varrho_1}$ and $T_{29} = T_{\varrho_2}$.
Now, assume $X$ is tropical resolvable in the sense of \cite[Def. 2.2]{BeHaHuNi2016}. Then there exists a regular refinement $\Sigma''$ of $\Sigma'$ giving rise to a toric variety $Z''$ such that
the Cox ring of
the proper transform $X''$ inside $Z''$
equals the freely graded ring
$\CC[T_{\varrho}; \ \varrho \in (\Sigma'')^{(1)}]/\bangle{\tilde{g}}.
$
This is a contradiction since 
$T_{\varrho_1}$
is not prime.
\end{example}

\section{3-dimensional canonical intrinsic quadrics of complexity two}\label{sec:quadrics}
In this section we treat the example class of {\em intrinsic quadrics}, 
i.e.\ varieties $X$ 
that admit a presentation of their Cox rings
$\RRR(X)$ by $\Cl(X)$-homogeneous generators, 
such that the ideal of relations is generated 
by a single quadratic polynomial.
Due to \cite[Prop 2.1]{FaHa2017} any intrinsic 
quadric
can be realized as an explicit general arrangement
variety $X=X(A,P,\Sigma) \subseteq Z$.
More precisely the defining relation of its Cox ring $R(A,P)$ is a quadric of the following form:
$$
g = 
T_{01}T_{02}
+ 
\ldots
+ 
T_{(q-1)1}T_{(q-1)2}
+
T_{q1}^2
+
\ldots
+ 
T_{r1}^2
\qquad
\text{with}
\quad
0 \leq q \leq r +1
$$

In particular, in the $\QQ$-Gorenstein case we can use the explicit description 
of the anticanonical complex to prove the classification result
of Theorem \ref{thm:quadrics}. 
As a direct consequence of this classification 
and the explicit bounds on the Picard number 
of $\QQ$-factorial Fano intrinsic quadrics
given in Proposition \ref{proposition:quadricsPicBound1},
we obtain in Proposition \ref{prop:quadricsTerminal} 
that there exists no $\QQ$-factorial Fano intrinsic quadric 
of complexity $c= \mathrm{dim}(X)- 1$ having at most terminal singularities.

Let us briefly recall the necessary facts about the combinatorial data encoding geometric properties of explicit general arrangement varieties. For a more comprehensive survey we refer to  \cite{HaHiWr2019}.

Let $X:=X(A,P,\Sigma) \subseteq Z$ be an explicit general arrangement variety with Cox ring $R(A,P)$ and total coordinate space $\bar{X}\subseteq\bar{Z}=\CC^{n+m}$, where we regard the right hand-side as a toric variety corresponding to the positive orthant $\gamma:=\QQ^{n+m}_{\geq 0}$.
Denote by $e_{ij}$ resp. $e_k$ the standard basis vectors of $\ZZ^{n+m}$ 
corresponding to the generators $T_{ij}$ resp. $S_k$ of $R(A,P)$. The matrix $P$ together with the degree map $Q\colon \ZZ^{n+m}\rightarrow K:=\Cl(X)$ sending $e_{ij}$ resp. $e_k$ to the classes $w_{ij}:=[T_{ij}]$ resp. $w_k:=[S_k]$ in $\Cl(X)$ fit into the following mutually dual exact sequences:
$$
\begin{tikzcd}
0\arrow[r]&L\arrow[r]&\ZZ^{n+m}\arrow[r,"P"]&N&\\
0&K\arrow[l]&\ZZ^{n+m}\arrow[l,"Q"]&M\arrow[l]&0\arrow[l]
\end{tikzcd}
$$
For a face $\gamma_0\preceq\gamma$ 
we denote by $\gamma_0^*\preceq\QQ^{n+m}_{\geq 0}$ the
complementary face generated by all $e_{ij}$ and $e_k$ which are not contained in $\gamma_0$.
Every face $\gamma_0\preceq\gamma$
defines a torus orbit $\bar{Z}(\gamma_0)\subseteq\bar{Z}$ corresponding to the cone $\gamma_0^*$. We call $\gamma_0$ an {\em $\bar{X}$-face} if the orbit $\bar{Z}(\gamma_0)$ intersects $\bar{X}$ non-trivially.
Now assume that $X$ is projective,
denote by $\mathrm{Ample}(X) \subseteq K_\QQ$ the cone of ample divisors and let $u\in\mathrm{Ample}(X)$ be any point. Then we call an $\bar{X}$-face $\gamma_0\preceq\gamma$ an {\em $X$-face} if the cone $Q(\gamma_0)\subseteq K_\QQ$ contains $u$ in its relative interior. 
This gives rise to a bijection 
$$\left\{X\text{-faces of } \gamma \right\}\rightarrow\left\{\text{cones in }\Sigma \right\},\qquad\gamma_0\mapsto P(\gamma_0^*).$$

In terms of the degree map and the $X$-faces, the cones of semiample and ample divisor classes of $X$ in $K_\QQ$ are given as 
$$
\SAmple(X)=\bigcap\limits_{\tiny
\begin{array}{c}
\gamma_0\preceq\gamma 
\\ 
X \text{-face}
\end{array}}Q(\gamma_0),
\qquad
\Ample(X)=\bigcap\limits_{\tiny
\begin{array}{c}
\gamma_0\preceq\gamma 
\\ 
X \text{-face}
\end{array}}Q(\gamma_0)^\circ.
$$

During this section we will make extensive use of the following characterization of $\QQ$-factoriality, see \cite[Prop. 5.4]{HaHiWr2019}.

\begin{remark}
Let $X:=X(A,P,\Sigma) \subseteq Z$ be a projective explicit general arrangement variety. Then the following statements are equivalent:
\begin{enumerate}
\item $X$ is $\QQ$-factorial.
    \item For every $X$-face $\gamma_0\preceq\gamma$ the image $Q(\gamma_0)$ is of full dimension inside $K_\QQ$.
    \item The semiample cone $\SAmple(X)$ is of full dimension inside $K_\QQ$.
\end{enumerate}
\end{remark}

\begin{lemma}\label{lem:quadricsTool}
Let $X:=X(A,P,\Sigma)\subseteq Z$ be a Fano intrinsic quadric. Then after suitably renumbering we may assume $n_0\geq\ldots\geq n_r$ and we are in one of the following situations:
\begin{enumerate}
    \item We have $n_{r-1}=n_r=1$ and
    $\cone(e_{(r-1)1},e_{r1},e_1,\ldots,e_m)$ is an $X$-face.
        \item We have $n_0=\ldots=n_{r-1}=2>n_r=1$ and $\cone(e_{01},e_{02},e_{r1},e_1,\ldots,e_m)$ is an $X$-face.
    \item We have $n_0=\ldots=n_r=2$ and 
    $\cone(e_{01},e_{02},e_{11},e_{12},e_1,\ldots,e_m)$ is an $X$-face.
\end{enumerate}
\end{lemma}
\begin{proof}
First note that in any of the above cases the cones under consideration are $\bar{X}$-faces.
To prove that they are indeed $X$-faces we show that their images in $K_\QQ$ contain $-\KKK_X$ in their relative interior.
Since all of the cones are pointed it suffices to show that $-\KKK_X$ can be written as a strictly positive combination over all extremal rays of the respective cone. For this let $n^{(1)}$ be the number of indices $i$ with $n_i=1$. Then we have
$$-\KKK_X=\frac{2r-n^{(1)}}{2}\deg(g)+\sum w_k\quad \text{ where }\quad 0<\frac{2r-n^{(1)}}{2}\leq r.$$
In case $(i)$ we have $\deg(g)=2w_{(r-1)1}=2w_{r1}$ 
and in case $(ii)$ we have $\deg(g)= w_{01}+w_{02}=2w_{r1}$,
which proves the assertion in these cases.
Finally, in case $(iii)$ we have $n^{(1)}=0$ and thus obtain $$-\KKK_X=r\deg(g)+\sum w_i = (w_{01}+w_{02})+(r-1)(w_{11}+w_{12})+\sum w_k.$$
\end{proof}

\begin{proposition}\label{proposition:quadricsPicBound1}
Let $X:=X(A,P,\Sigma)\subseteq Z$ be a $\QQ$-factorial, Fano intrinsic quadric. Then $\varrho(X)\leq 3+m$ holds.
\end{proposition}

\begin{proof}
We distinguish between the three cases treated in Lemma \ref{lem:quadricsTool} and show that the dimension of the $X$-faces occurring there
is at most $m+3$. Then using $\QQ$-factoriality of $X$,
we obtain the bound on $\varrho(X)$ as claimed.

In Case (i) of Lemma \ref{lem:quadricsTool} we obtain an $X$-face of dimension at most $1+m$ as $w_{r1} = w_{(r-1)1}$  holds
due to homogeneity of the defining relation $g$.
Similar, in Case (ii) we obtain an $X$-face of dimension at most $2+m$ as
$w_{r1} \in \mathrm{cone}(w_{01}, w_{02})$
holds. 
Consider Case (iii). 
Here we have $$w_{01} + w_{02} - w_{11} - w_{12} = \deg(g) - \deg(g) = 0.$$
In particular, the cone  $\mathrm{cone}(w_{01},w_{02},w_{11}, w_{12}, w_1, \ldots, w_m)$ is of dimension at most $3+m$. This completes the proof. 
\end{proof}

\begin{corollary}\label{cor:quadricsPicNum}
Let $X:=X(A,P,\Sigma)\subseteq Z$ be a $\QQ$-factorial, Fano intrinsic quadric of complexity $c=\dim(X)-1$. Then $\varrho(X)\leq 5$ holds.
\end{corollary}
\begin{proof}
In this situation we have a one-dimensional lineality space $\lambda_{\lin}$. 
In particular, we have $s=1$
for the lower part of the matrix $P$. This implies $m\leq 2$ as the columns of $P$ are assumed to be pairwise different and primitive.
\end{proof}

\begin{proposition}\label{prop:quadricsPicardNumber}
Let $X:=X(A,P,\Sigma)\subseteq Z$ be a three-dimensional  $\QQ$-factorial
Fano intrinsic quadric of complexity $c=2$.
Then $\varrho(X)\leq 3$ holds.
Moreover, if $X$ is a full intrinsic quadric, then $\varrho(X) = 1$ holds.
\end{proposition}

\begin{proof}
Due to Corollary \ref{cor:quadricsPicNum} we have $\varrho(X)\leq 5$. Assume $\varrho(X)\geq 4$. We go through the possible configurations of $n=n_0+\ldots+n_3$ and $0\leq m\leq 2$.
After renumbering the columns of $P$ we arrive at one of the following cases: \begin{enumerate}
    \item $n_0=\ldots=n_2=2>n_3 = 1,\ m=1$
    \item $n_0=n_1=2>n_2=n_3 = 1,\ m=2$
    \item $n_0=\ldots=n_3=2,\ m=1$
    \item $n_0=\ldots=n_2=2>n_3 = 1,\ m=2$
\end{enumerate}
\noindent In the Cases $(i)$ and $(ii)$ we have $\varrho(X)=4$ and in the Cases $(iii)$ and $(iv)$, we have $\varrho(X)=5$.
Applying Lemma \ref{lem:quadricsTool} $(ii)$ in the Cases $(i)$ and $(iv)$ we obtain a three-dimensional $X$-face which contradicts $\QQ$-facatoriality of $X$. 
Similar, applying Lemma \ref{lem:quadricsTool} $(i)$ in the Cases $(ii)$ and Lemma \ref{lem:quadricsTool} $(iii)$ in the Case $(iii)$ we obtain a three-dimensional $X$-face and thus a contradiction to $\QQ$-factoriality as well.
For the supplement let $X$ be a full intrinsic quadric and assume $\varrho(X)>1$. 
Due to Proposition \ref{proposition:quadricsPicBound1} we have $\varrho(X)\leq 3$. Thus renumbering the columns of $P$ we are left with the following situations:
\begin{enumerate}
    \item $n_0=n_1=2>n_2=n_3=1$
    \item $n_0=n_1=n_2=2>n_3=1$
\end{enumerate}
In Case $(i)$ we have $\varrho(X)=2$ and in Case $(ii)$ we have $\varrho(X)=3$. Using the same argument as before we exclude Case $(i)$ using Lemma \ref{lem:quadricsTool} (i) and Case $(ii)$ using Lemma \ref{lem:quadricsTool} (ii).
\end{proof}

\begin{remark}\label{remark:quadricsConstellations}
Let $X:=X(A,P,\Sigma)\subseteq Z$ be a $\QQ$-factorial Fano intrinsic quadric of dimension three and
torus action of complexity two.
Then the dimension of the total coordinate space and the Picard number are given as: $$\dim(\overline{X}) = n+m-1,\quad \varrho(X)=n+m-4, \quad 4\leq n\leq 8,\quad 0\leq m\leq 2.$$
\end{remark}

\begin{remark}\label{rem:quadricsLinPartBound}
Let $X:=X(A,P,\Sigma)\subseteq Z$ be a Fano explicit general arrangement variety with anticanonical complex $\mathcal{A}$. Then every convex combination of vertices of $\mathcal{A}$ that lie inside the tropical variety $\trop(X)\subseteq\QQ^{r+s}$ is contained in $\AAA$. 
In particular, if $X$ is of dimension three, has a torus action of complexity two and at most canonical singularities, then for any such point $v=(v_1,\ldots,v_{r+1})$ that lies inside the lineality space of the tropical variety $\trop(X)$, we have $-1\leq v_{r+1}\leq 1$.
\end{remark}

Let $X$ be a Mori dream space.
In order to detect a maximal torus action
on $X$ we will make use of the procedure of {\em lifting} automorphisms. 
Assume there is a torus action
$\TT\times X\rightarrow X$. Then due to \cite[Thm. 4.2.3.2]{ArDeHaLa2015} there is a lifted action $\TT\times\hat{X}\rightarrow\hat{X}$ with 
$$
t\cdot (h\cdot \hat x) = h\cdot (t\cdot \hat x)\qquad \text{ for all } t\in\TT,
\
h\in H_X, 
\
\hat x\in\hat X.
$$
Thus $\TT$ as well as the product $\TT\times H_X$ act on $\hat{X}$ and therefore on $\bar{X}$ as $\hat{X}\subseteq\bar{X}$ is of codimension two.
We will identify both groups with the corresponding subgroups of translations inside the automorphism group $\mathrm{Aut}(\bar X)$.

We will show that in our situation this action is diagonal in the following sense:
Let $X\subseteq \CC^n$ be an affine variety endowed with an effective quasitorus action ${H\times X\rightarrow X}$. 
We say that $H$ \emph{acts diagonally} on $X$ if
there are characters ${\chi^{w_1},\ldots,\chi^{w_n}\in\Chi(H)}$ 
such that $h\cdot(x_1,\ldots,x_n) = (\chi^{w_1}(h)x_1,\ldots,\chi^{w_n}(h)x_n)$ holds for all $h\in H$ and $(x_1,\ldots,x_n)\in X$. Note that this is equivalent to
homogeneity of the coordinate functions $T_i\in \OOO(X)$, where we endow $\OOO(X)$ with the grading 
corresponding to the action of $H$ on $X$.

\begin{lemma}\label{lem:auto}
Let $X$ be a Mori dream space with torus action $\TT\times X\rightarrow X$ and Cox ring $\RRR(X)=\CC[T_1,\ldots,T_r]/\bangle{g_1,\ldots,g_s}$. If all homogeneous components $\RRR(X)_{w_i}$ with $w_i=\deg(T_i)$ are one-dimensional then $\TT\times H_X\subseteq \Aut(\bar{X})$ acts diagonally.
\end{lemma}
\begin{proof}
Consider the grading on $\mathcal{R}(X)$ 
defined by the action $\TT \times H_X$ on $\bar X$.
Then the $K_X$-grading is a coarsening of this grading and thus, as the homogeneous components $\mathcal{R}(X)_{w_i}$ are one-dimensional, they are homogeneous components in this refined grading as well. Thus the assertion follows.
\end{proof}

\begin{proof}[Proof of Theorem  \ref{thm:quadrics}]
According to Proposition \ref{prop:quadricsPicardNumber} we have $\varrho(X)\leq 3$. 
We first go through the cases sorted by the Picard number and then prove that none of the varieties in the list are isomorphic.
Finally, we determine the dimension of the maximal torus in their automorphism groups. 

\vspace{2mm}\noindent
{\em Case (I) ($\varrho(X)=1$):}
Due to Remark \ref{remark:quadricsConstellations} we are left with the following configurations:
\begin{enumerate}
    \item[(a)] $n=4$ and $m=1$
    \item[(b)] $n=5$ and $m=0$
\end{enumerate}

\vspace{2pt}
\noindent
{\em Case (I)(a):}
As for $\varrho(X)=1$ any $\bar{X}$-face is an $X$-face, we obtain a big cone $\sigma=\cone(v_{01},v_{11},v_{21},v_{31})\in\Sigma$. After applying suitable row operations on the matrix $P$, we may assume
$$
    P=\left[\begin{array}{ccccc}
    -2&2&0&0&0\\
    -2&0&2&0&0\\
    -2&0&0&2&0\\
    x&1&1&1&-1
    \end{array}\right], 
    \qquad
    v_\sigma'=\left[0,0,0,\frac{x+3}{2}\right].
$$
As $X$ has at most canonical singularities, we conclude $0< (x+3)/2\leq 1$ and thus $x= -1$ as the 
columns of $P$ are primitive. The resulting variety $X(A,P,\Sigma)$ is canonical and appears as No. 1 in our list.

\vspace{2pt}
\noindent
{\em Case (I)(b):} After suitably renumbering we may assume $n_0=2$
and with $\varrho(X)=1$ we obtain the following two big cones in $\Sigma$:
$$\sigma_j:=\cone(v_{0j},v_{11},v_{21},v_{31}), \quad \text{ where } 1\leq j\leq 2.$$
Moreover, after applying suitable row operations, the matrix $P$ and the vectors $v_{\sigma_1}'$ and $v_{\sigma_2}'$ are of the following form:
$$
    P=\left[\begin{array}{ccccc}
    -1&-1&2&0&0\\
    -1&-1&0&2&0\\
    -1&-1&0&0&2\\
    x&y&1&1&1
    \end{array}\right], 
    \qquad
    \begin{array}{ccc}
    v_{\sigma_1}'&=&\left[0,0,0,\frac{2x+3}{3}\right],
    \\
    \\
    v_{\sigma_2}'&=&\left[0,0,0,\frac{2y+3}{3}\right],
    \end{array}
$$
where $x<y$ holds. As $X$ has at most canonical singularities, we conclude $0< (2x+3)/3\leq 1$ and $-1\leq (2y+3)/3< 0$.  This implies $x\in\{-1,0\}$ and $y\in\{-3,-2\}$.
In this situation all of the possible varieties $X(A,P,\Sigma)$ 
are canonical. 
Note that for $x=-1$, $y=-3$ and $x= 0$, $y=-2$ the resulting rings $R(A,P)$ are isomorphic. All in all this gives the varieties Nos.\  2 to 4 in our list.

\vspace{2pt}
\noindent
{\em Case (II) ($\varrho(X)=2$):}
Due to Remark \ref{remark:quadricsConstellations} we are in one of the following situations:
\begin{enumerate}
    \item[(a)] $n=4$ and $m=2$
    \item[(b)] $n=5$ and $m=1$
    \item[(c)] $n=6$ and $m=0$
\end{enumerate}

\vspace{2pt}
\noindent
{\em Case (II)(a):}
After applying suitable row operations, we arrive at
$$
    P=\left[\begin{array}{cccccc}
    -2&2&0&0&0&0\\
    -2&0&2&0&0&0\\
    -2&0&0&2&0&0\\
    x&1&1&1&-1&1
    \end{array}\right].
$$
We obtain a point
$$(0,0,0,\frac{x+3}{4}) = \frac{1}{4}(v_{01} + v_{11}+ v_{21}+v_{31})\in\conv(v_{01},v_{11},v_{21},v_{31})\cap|\trop(X)|.$$
Remark \ref{rem:quadricsLinPartBound} implies $-1\leq 1/4(x+3)\leq 1$ and thus
$-7\leq x \leq 1$. 
Assume $x$ is even. Then the first column of $P$ is not primitive; a contradiction.
For $x\in\{-7,-5,-1,1\}$ the resulting varieties are not Fano. 
Thus, the only possible case left is $x=-3$. The resulting variety is a canonical Fano variety and appears as No. 5 in our list.

\vspace{2pt}
\noindent
{\em Case (II)(b):}
We may assume $n_0$=2 and
after applying suitable row operations we arrive at
$$
    P=\left[\begin{array}{cccccc}
    -1&-1&2&0&0&0\\
    -1&-1&0&2&0&0\\
    -1&-1&0&0&2&0\\
    x&y&1&1&1&-1
    \end{array}\right],
$$
where we may assume $x\leq y$. Note that due to completeness of $X$, we have $|\trop(X)|\subseteq|\Sigma|$. Therefore, we obtain a big cone $\sigma$ containing $[0,0,0,1]$ and a vertex $v_\sigma'$ of $\AAA$
$$\sigma=\cone(v_{02},v_{11},v_{21},v_{31}),\qquad v_\sigma'=[0,0,0,1+\frac{2}{3}y].$$
Due to canonicity of $X$ we conclude $0\leq 1+ (2/3)y\leq 1$
and thus $-1\leq y\leq 0$.
Now consider the point 
$$[0,0,0,\frac{2x+3}{5}] = \frac{1}{5}(2 v_{01} + v_{11}+ v_{21}+v_{31})\in\conv(v_{01},v_{11},v_{21},v_{31})\cap|\trop(X)|.$$
Using Remark \ref{rem:quadricsLinPartBound} we obtain 
$-1\leq 1/5(2x+3)\leq 1$ and thus $-4\leq x\leq 1$. 
Computing the anticanonical complex in these cases gives Nos.\  6 to 8 in our list.

\vspace{2pt}
\noindent
{\em Case (II)(c):}
In this case $X$ is a full intrinsic quadric and therefore Proposition \ref{prop:quadricsPicardNumber} implies $\varrho(X)=1$, a contradiction to $\varrho(X)=2$.

\vspace{2pt}
\noindent
{\em Case (III) ($\varrho(X)=3$):}
Due to Remark \ref{remark:quadricsConstellations} and Lemma \ref{lem:quadricsTool} $(i)$ we may assume $2=n_0>n_1=\ldots=n_3=1$ and $m=2$.
After applying suitable row operations we arrive at
$$
    P=\left[\begin{array}{ccccccc}
    -1&-1&2&0&0&0&0\\
    -1&-1&0&2&0&0&0\\
    -1&-1&0&0&2&0&0\\
    x&y&1&1&1&-1&1
    \end{array}\right]
$$
where we may assume $x\leq y$.
Consider the point 
$$[0,0,0,\frac{2x+3}{5}] = \frac{1}{5}(2 v_{01} + v_{11}+ v_{21}+v_{31})\in\conv(v_{01},v_{11},v_{21},v_{31})\cap|\trop(X)|.$$
Using Remark \ref{rem:quadricsLinPartBound} we obtain 
$-1\leq 1/5(2x+3)\leq 1$ and thus $-4\leq x\leq 1$. 
Replacing $v_{01}$ with $v_{02}$ we obtain $-4\leq y\leq 1$ as well.
Computing the anticanonical complex in these cases gives No. 9 in our list.

We proceed by proving that the varieties defined by the data in our list are pairwise non-isomorphic.

Due to the divisor class group the only possible combinations to compare are Nos.\  2 and 3 and Nos.\  6, 7 and 8.
The Fano index of a Fano explicit general arrangement variety $X=X(A,P,\Sigma)$
is the largest integer $q(X)$ such that $-\KKK_X=q(X) w$ holds with some $w\in\Cl(X)$.
If $X$ is isomorphic to another general arrangement variety $X'$ then their Fano indices coincide. 
Denote by $X_i$ the explicit general arrangement variety defined by the $i$-th datum in our list. Then the varieties $X_2$ and $X_3$, $X_6$ and $X_8$ and $X_7$ and $X_8$ are not isomorphic due to the following table:
\begin{center}
\begin{tabular}{c|c|c|c|c|c}
$X_i$&$X_2$&$X_3$&$X_6$&$X_7$&$X_8$
\\\hline
$q(X_i)$&$3$&$6$&$1$&$1$&$2$
\end{tabular}
\end{center}
Thus, we are left with comparing Nos.\  6 and 7.
The effective cone of an explicit general arrangement variety $X=X(A,P,\Sigma)$ is the cone 
$$\Eff(X)=\cone(w_{ij},w_k,\ 0\leq i\leq r, 1\leq j\leq n_i, 1\leq k\leq m).$$
If $X$ is isomorphic to another general arrangement variety $X'$ then there is a lattice isomorphism mapping the extremal primitive ray generators of $\Eff(X)$ onto that of $\Eff(X')$.
Considering the varieties $X_6$ and $X_7$ their effective cones are $\Eff(X_6) = \cone([1,-1],[0,1])$ and $\Eff(X_7)=\cone([-1,2],[1,0])$. Thus the varieties are not isomorphic 
as $\Eff(X_6)$ is a smooth cone whereas $\Eff(X_7)$ is not.

We finish the proof by showing that the varieties in our list are of true complexity two. 
We treat all varieties except 
the one encoded by the 2nd datum 
of our list at once. 
Let $X$ be any of these varieties 
and assume $\TT\times X\rightarrow X$ is a maximal torus action on $X$.
As all the homogeneous components $\RRR(X)_{w_i}$ are one-dimensional Lemma \ref{lem:auto} applies and we conclude that in these cases the generators $T_i$ resp. $S_k$ are homogeneous with respect to the grading defined via the $(\TT \times H_X)$-action on $\bar{X}$. As 
the ideal defining $\bar{X}$ is principle, the relation is homogeneous as well and thus $\TT\times H_X$ acts as a sub-quasitorus of the maximal quasitorus defined via the maximal diagonal grading, see \cite[Constr. 3.2.4.2]{ArDeHaLa2015}.
Modding out the $H_X$-action yields that $\TT$ is indeed one-dimensional. 

Now, let $X$ be the variety encoded by the 2nd datum in our list. Here, the homogeneous components $\mathcal{R}(X)_{w_i}$ are one-dimensional for $i \geq 3$. In particular, the 
variables $T_3, T_4$ and $T_5$ are homogeneous with respect to the $(\ZZ^t \times K_X)$-grading defined via the 
$(\TT \times H_X)$-action on $\bar X$.
Considering the $2$-dimensional graded component $\mathcal{R}(X)_{w_1} = \mathcal{R}(X)_{w_2}$
one concludes that there exists a
$(\ZZ^t \times K_X)$-homogeneous set of generators of the form
$
T_1 + \lambda T_2, \ \mu T_1 + T_2, \ T_3, \ T_4$ and $T_5 \in \mathcal{R}(X),$
where $\lambda, \mu \in \CC$. 
We obtain a graded isomorphism between $\mathcal{R}(X)$ and 
$$R:=\CC[f_1, f_2, f_3, f_4, f_5]/\bangle{\mu f_1^2 - (\lambda \mu +1)f_1f_2 + \lambda f_2^2 - (\lambda\mu - 1)^2(T_3^2 + T_4^2 + T_5^2)},
$$
where the variables and the relation of the latter ring are even $(\ZZ^t \times K_X)$-homogeneous. We conclude that
the $(\ZZ^t \times K_X)$-grading on $R$ is a coarsening of its maximal diagonal grading. Modding out $K_X$, we conclude $t \leq 1$ and $\TT$ is indeed one-dimensional.
\end{proof}

\begin{proposition}\label{prop:quadricsTerminal}
Let $X=X(A,P,\Sigma)\subseteq Z$ be a $\QQ$-factorial Fano intrinsic quadric of complexity $c=\dim(X)-1$. Then $X$ is not terminal.
\end{proposition}

\begin{proof}
Denote by $n^{(1)}$ resp. $n^{(2)}$ the number of terms of $g$ with $n_i = 1$ resp. $n_i=2$, where $0\leq i\leq r$.
Then,
as $X$ is $\QQ$-factorial and of complexity $c=\dim(X)-1$, the dimension and the Picard number of $X$ are given
as
$$
\dim(X) = c + 1 = n^{(1)} + n^{(2)} -1,
\qquad
\varrho(X) = n + m - r - 1 =  n^{(2)} + m.
$$
In particular, using 
Proposition \ref{proposition:quadricsPicBound1} we conclude
$\dim(X) = \varrho(X) - m + n^{(1)} -1 \leq n^{(1)} +2$.
In case that $X$ is of dimension two, terminality means smoothness and the assertion follows due to the classification of smooth Del Pezzo surfaces, see \cite{De1980, Ma1966}. In case that $X$ is of dimension $3$, Theorem \ref{thm:quadrics} shows that
there exist no terminal varieties.
Assume $\dim(X) \geq 4$. Then 
$n^{(1)}\geq 2$ holds 
and after reordering and applying 
admissible row operations
we may assume that $P$ contains the following two columns:
$$v_{(r-1)1} = (0, \ldots, 0, 2,0,1),
\qquad
v_{r1} = (0, \ldots, 0,2,1).
$$
As $X$ is complete and $c \geq 2$ holds we have
$\mathrm{cone}(v_{(r-1)1},v_{r1})\in\Sigma$ 
and Remark \ref{rem:quadricsLinPartBound} implies
$$
(0,\ldots, 0,1,1,1) \in \mathrm{conv}(v_{(r-1)1},v_{r1}) \subseteq \mathcal{A}.
$$
In particular, there is a lattice
point in $\mathcal{A}$ that is not a primitive ray generator of $\Sigma$
and therefore $X$ can not be terminal.
\end{proof}

{\small
\bibliography{references}{}
\bibliographystyle{plain}
}
\end{document}